\documentclass{article}
\usepackage{authblk,color,amsfonts,amsmath,mathabx}
\usepackage[a4paper, margin=0.6in]{geometry}

\usepackage{bm}
\usepackage{amsthm}
\usepackage[format=plain,labelfont=bf,up,textfont=it,up]{caption}
\usepackage{graphicx,subfigure,epstopdf,wrapfig,chemarrow}
\usepackage{tabularx}
\usepackage[ruled,linesnumbered]{algorithm2e}
\SetKwRepeat{Do}{do}{while}
\usepackage{url}

\usepackage{setspace}
\usepackage{lineno}
\newtheorem{theorem}{Theorem}

\newtheorem{lemma}{Lemma}
\newtheorem{proposition}{Proposition}
\newtheorem{example}{Example}

\usepackage{tikz,array}
\usetikzlibrary{shapes.geometric}
\usetikzlibrary{shapes.arrows}

\begin{document}

\title{A convergence analysis of a structure-preserving gradient flow 
method for the all--electron Kohn-Sham model}

\author[1]{Yedan Shen} \affil[1]{School of Mathematics and Information Science, Guangzhou University, Guangzhou, China.}

\author[2]{Ting Wang} \affil[2]{Faculty of Science and Technology, University of Macau, Macao SAR, China.}

\author[3,4]{Jie Zhou} \affil[3]{Hunan Key Laboratory for Computation and Simulation in Science and Engineering, Xiangtan University, Xiangtan, China.} \affil[4]{School of Mathematical Computational Sciences, Xiangtan University, Xiangtan, China.}

\author[2,5,6,\footnote{Corresponding author. Email: {\it
      garyhu@um.edu.mo}}]{Guanghui Hu} \affil[2]{Faculty of Science and Technology, University of Macau, Macao SAR, China.}\affil[5]{Zhuhai UM Science \& Technology Research Institute, Zhuhai, China.} \affil[6]{Guangdong-Hong Kong-Macao Joint Laboratory for Data-Driven Fluid Mechanics and Engineering Applications, University of Macau, Macao SAR, China.}

\date{}

\maketitle

\begin{abstract}
  In [Dai et al, Multi. Model. Simul., 2020], a structure-preserving
  gradient flow method was proposed for the ground state calculation
  in Kohn--Sham density functional theory, based on which a linearized
  method was developed in [Hu, et al, EAJAM, accepted] for further
  improving the numerical efficiency.  In this paper, a complete
  convergence analysis is delivered for such a linearized method for
  the all-electron Kohn--Sham model. Temporally, the convergence, the
  asymptotic stability, as well as the structure-preserving property
  of the linearized numerical scheme in the method is discussed
  following previous works, while spatially, the convergence of the
  $h$-adaptive mesh method is demonstrated following [Chen et al,
  Multi. Model. Simul., 2014], with a key study on the boundedness of
  the Kohn--Sham potential for the all-electron Kohn--Sham
  model. Numerical examples confirm the theoretical results very well.

\end{abstract}




\section{Introduction}

The Kohn--Sham density functional theory (KSDFT) proposed in 1965 is
one of the most successful approximation models towards the
computational quantum chemistry, condensed matter physics, etc., for
many-body electronic structure calculations
\cite{fiolhais2003primer,zhang2017adaptive}. Due to the nonlinearity
of the governing equation and the complexity of the given electronic
structure system, obtaining a high-quality numerical solution has been
becoming an important issue in the simulation.

For the ground state calculation, concerned with the nonlinearity of
the Kohn--Sham equation caused by the Hamiltonian operator, the
\emph{self consistent field} (SCF) iteration is a classical choice
\cite{lin2013elliptic,kuang2020stabilizing}. Several types of
numerical methods are employed in the simulation of Kohn--Sham
equation, such as finite element methods
\cite{bao2015real,chen2014adaptive,dai2020gradient}, finite difference
methods \cite{tancogne2020octopus}, spectral methods
\cite{motamarri2014subquadratic,guo2021mortar}, discontinuous Galerkin
methods \cite{lin2012adaptive}, finite volume methods
\cite{dai2011finite}, plane--wave methods
\cite{genovese2011daubechies}, etc. Besides directly solving the
Kohn--Sham equation, minimizing the total energy is also a
popular way to obtain the ground state of the quantum system. The
ground state of the system can be obtained by solving the following
minimization model with orthogonality constraints
\cite{payne1992iterative},

\begin{equation}
  \left \{
  \begin{array}{l}
  	\displaystyle \min_{U} E_{KS}(U), \\
  	 \displaystyle  \\
	 \displaystyle \langle U, U \rangle=I_p.
 \end{array}
  \right .
\end{equation}

There have been several types of optimization methods proposed for the
Kohn--Sham energy minimization model such as the quasi-Newton methods
\cite{pfrommer1997relaxation}, the constrained minimization methods
\cite{dai2021parallel,wu2005direct}, the conjugate gradient methods
\cite{dai2017conjugate,payne1992iterative}. A main advantage of this
approach is that the solution of nonlinear eigenvalue problems can be
avoided, based on which the main cost becomes the assembling of the
total energy functional and operations on the manifold. It should be
pointed out that the orthonormalization has to be invoked explicitly
or implicitly by using most of the existing optimization methods.
However, implementing such orthogonality-preserving strategies would
occupied a large part of the CPU time, which makes it a tough task for
the ground state calculation of a large scale quantum system. Thus,
concerned with the efficiency and the parallel scalability, an
orthogonalization-free model is needed towards the simulation of large
scale system. It is worth mentioning that in
\cite{gao2022orthogonalization}, an infeasible approach has been
proposed based on the finite element method for calculating the ground
state, which successfully removed the orthogonalization operation in
the simulation.

Towards the same purpose, another approximation to avoid the
orthogonalization is the gradient flow method. Recently, in
\cite{dai2020gradient}, Dai, etc., have introduced and analyzed a
Kohn--Sham gradient flow based model for electronic structure
calculations, in which an extended gradient flow has been proposed. It
has shown that the orthonormalization relation among those
wavefunctions can be automatically preserved during the simulation on
a Stiefel manifold. Concerned with the numerical scheme in
\cite{dai2020gradient}, an implicit midpoint scheme has been employed
for temporal discretization, based on which the effectiveness of the
algorithm has been demonstrated successfully. Motivated by
\cite{dai2020gradient}, a linear scheme which is designed to improve
the computational efficiency is proposed in \cite{hu2022}, in which
several adaptive strategies also haven been adopted to accelerate the
simulations in order to handle the nonlinearity of the governing
equation and the singularity generated by the external potential. It
is noted that, although structure-preserving properties and attractive
numerical performance of the proposed gradient flow methods in
\cite{hu2022} have been successfully
demonstrated, a theoretical study of the numerical convergence of the
class of the methods is still missing.

In this paper, concerned with the above issue, first and foremost, a
complete analysis for the algorithm in \cite{hu2022} is delivered, of
which the convergence has been verified both temporally and spatially
in an unified framework. Spatially, motivated by
\cite{chen2014adaptive}, the convergence of the adaptive finite
element approximation (AFE) of the gradient flow model is demonstrated
in this paper. By contrast, the convergence analysis is demonstrated
for the all-electron Kohn--Sham gradient flow model in this paper, in
which the boundedness of the Kohn--Sham total energy
\cite{shen2019asymptotics} plays an important role.  Temporally,
following \cite{dai2020gradient}, the convergence of the proposed
numerical scheme has been verified theoretically, in which the
asymptotically stability is shown to be valid for our linearized
structure-preserving numerical scheme as well.  Finally, the
convergence of the proposed numerical method is demonstrated by some
numerical examples, which validate our theoretical results.
Furthermore, towards improving the efficiency of the simulation, the
adaptive finite element methods \cite{chen2021decoupled}
could be employed to efficiently achieve the ground state of the
Kohn--Sham system as well.  The adaption methods in the electronic
structure calculations mainly include the r-adaptive methods
\cite{bao2013numerical}, $h$-adaptive methods \cite{chen2014adaptive,
  shen2019asymptotics}, hp-adaptive methods \cite{torsti2006three}.
In this paper, the $h$-adaptive finite element method has been
adopted, which would be discussed in details in the following
contents.

The rest of this paper is organized as follows: In Section 2, the
KSDFT model and some necessary notations are introduced. In Section 3,
the Kohn--Sham gradient flow model with a fully discretization are
described in details, and an algorithm with adaptive strategies is
proposed.  In Section 4, the convergence of our linearized numerical
scheme for the gradient flow model is verified both temporally and
spatially.  In Section 5, some numerical experiments are presented to
verify the convergence of our method. The conclusion is given in
Section 6 finally.

\section{Preliminaries}

\subsection{Some notations}

For convenience, some notations used in the following discussions will
be introduced firstly in this subsection.
The standard $L^2$-inner product $(\cdot,\cdot)$ is applied 
in this paper, which is denoted by
\begin{equation}
	(u,v)=\int_{\Omega} u(x)v(x)dx,
\end{equation}
with associated $L^2$ norm defined as
$\lVert u \rVert=(u,u)^{\frac{1}{2}}$, additionally, the $H^1-$norm is
defined as follows
\begin{equation}
	\lVert u \rVert_1=\sqrt{\lVert u \rVert^2+\lVert \nabla u \rVert^2},
\end{equation}
where $H^1(\Omega)$ is the standard Sobolev space, and 
$H_0^1(\Omega)=\{u\in H^1(\Omega): u|_{\partial \Omega}=0\}$.

By means of the above notations, we define $\mathcal{H}=(H_0^1(\Omega))^N$, 
and $U=(u_1,u_2,\cdots,u_N)$, 
$\Phi=(\phi_1,\phi_2,\cdots, \phi_N) \in \mathcal{H}$. 
Then the inner product matrix takes the following expression,
\begin{equation}
	\langle U^T \phi \rangle=(u_i, \phi_j)_{i,j=1}^N 
	\in \mathbb{R}^{N \times N},
\end{equation}
where the norm of $U$ is defined by
\begin{equation}
	\lvert \lvert \lvert U \rvert \rvert \rvert=(U,U)^{\frac{1}{2}},
\end{equation}
where $(U,U)=tr \langle U^T \Psi\rangle$, and the notation $tr$
means the trace of a matrix.

Finally, we introduce the Stiefel manifold , which takes the 
following form
\begin{equation}
	\mathcal{M}^N=\{U \in \mathcal{H}: \langle U^T U \rangle=I_N \},
\end{equation}
where $I_N$ is an identity matrix of size $N$. Moreover, as mentioned
in \cite{dai2020gradient}, For $U \in (H^1(\mathbb{R}^3))^N$ and 
any matrix $P \in \mathbb{R}^{N \times N}$, there holds
\begin{equation}
	UP= \Big( \sum_{j=1}^N p_{j1}u_j,\sum_{j=1}^N p_{j2}u_j,
	\cdots,\sum_{j=1}^N p_{jN}u_j\Big),
\end{equation}
where
\begin{equation}
	U \in \mathcal{M}^N \Leftrightarrow UP \in \mathcal{M}^N,
	~~ \forall ~P \in \mathcal{O}^N,
\end{equation}
with $\mathcal{O}^N=\{P \in \mathbb{R}^{N \times N}: P^TP=I_N\}$.

Meanwhile, an equivalent relation $"\sim"$ is defined on 
$\mathcal{M}^N$ as
\begin{equation}
	U \sim \hat{U} \Leftrightarrow \exists ~ P \in \mathcal{O}^N,
	~\hat{U}=UP,
\end{equation}
and get a Grassmann manifold, which is a quotient of $\mathcal{M}^N$:
\begin{equation}
	G^N=\mathcal{M}^N / \sim.
\end{equation}
Furthermore, the equivalent class of $U \in \mathcal{M}^N$ is denoted
as follows,
\begin{equation}
	[U]=\{UP: P \in \mathcal{O}^N\}.
\end{equation}

\subsection{KSDFT model}
Aiming at obtaining the ground state of a given molecular system in 
all-electron calculation, the Kohn--Sham density functional theory (KSDFT)
model has been introduced to achieve this goal, which can be stated 
as: Find $(\lambda_i,u_i) \in \mathbb{R}^- \times H_0^1(\Omega)$ , 
$i=1,\cdots,N $, such that
\begin{equation}\label{kse}
	\left \{
	\begin{array}{l}
	\displaystyle \hat{H}u_i=\lambda_iu_i,\\
	\displaystyle \\
	\displaystyle (u_i,u_j)=\delta_{ij},\\
	\end{array}
	\right .	
\end{equation}
where $\hat{H}$ is the Hamiltonian operator consisting of two parts, the
kinetic potential part $-\nabla^2/2$ and the effective potential part
$V_{eff}([\rho];\bm{r})$ which takes the following expression,
\begin{equation}
	V_{eff}([\rho];\bm{r})=V_{ext}(\bm{r})+V_{Har}([\rho];\bm{r})
	+V_{xc}([\rho];\bm{r}),
\end{equation}
where $\rho$ is the electronic density.

The total energy of the given KSDFT system consists of several parts:
\begin{equation}
	E_{KS}=E_{kinetic}+E_{ext}+E_{Har}+E_{xc}+E_{nuc},
\end{equation}
where $E_{kinetic}$ is the kinetic energy, and $E_{ext}$, $E_{Har}$, $E_{xc}$,
and $E_{nuc}$ denote the external energy, Hartree energy, exchange-correlation
energy, and the nucleus-nucleus energy, respectively, of which the formulation
are given as follows,
\begin{align}
	E_{kinetic} & =\frac{1}{2}\sum_{l=1}^p \int_{\mathbb{R}^3} \lvert \nabla u_l\rvert^2 d\bm{r}, &  E_{ext} & = \int_{\mathbb{R}^3} V_{ext} \rho(\bm{r}) d \bm{r}, & E_{Har} & =\frac{1}{2}\int_{\mathbb{R}^3} 
	V_{Har}\rho(\bm{r}) d \bm{r}, \nonumber \\ \nonumber
	E_{xc} & = \int_{\mathbb{R}^3} \epsilon_{xc}\rho(\bm{r}) d \bm{r}, &
	E_{nuc} & =\sum_{j=1}^M \sum_{k=j+1}^M \frac{Z_j Z_k}{\lvert \bm{R}_j-\bm{R}_k\rvert}.
\end{align}

Other than solving the solving the Kohn--Sham equation, 
the ground state solution can be obtained by minimizing the total energy
with orthogonality constraints \cite{payne1992iterative}:
\begin{eqnarray}
	&& \min_{U} E(U) \\ \nonumber
	&& \mbox{s.t.} \langle U,U\rangle=I_p. 
\end{eqnarray}

Towards the minimization of the Kohn--Sham total energy, a class of 
optimization methods can be applied. One of them is the 
gradient type method, 
which are employed in \cite{dai2017conjugate,zhang2014gradient}, 
etc., with several orthogonalization technologies.
However, the orthogonalization process invoked in these methods may
cause high complexity and low parallel scalability when solving large
quantum systems. To breakthrough this bottleneck, an orthogonality
preserving scheme defined on the Stiefel manifold has been proposed
in \cite{dai2020gradient}, which solves the following minimization 
problem,
\begin{equation}\label{minipro}
	\inf_{U \in \mathcal{M}^N} E(U),
\end{equation}
in which an orthonormalization free method is described and the
excellent results are obtained.

Motivated by their idea, a linearized orthogonal-preserving numerical 
scheme for the Kohn--Sham gradient flow model would be employed in 
this paper, of which the convergence will be demonstrated both 
temporally and spatially in a unified scheme in the following contents.

\section{The gradient flow based KSDFT model}
\subsection{Gradient flow model for ground state calculation}

To avoid the orthogonalization operation in the 
algorithm, several types of methods has been employed towards obtaining 
the ground state in KSDFT, one of them is the gradient flow model 
proposed in \cite{dai2020gradient}, which is described as follows:

The idea of gradient flow method is to find a function $U(t)$, such that the
critical point $U^*$ for the minimization problem can be reached when
$t \rightarrow \infty$. Therefore, the curve $U=U(t)$ should satisfy the
following initial value problem,
\begin{equation}
\left \{
\begin{array}{l}
	\displaystyle \frac{dU}{dt}=-\nabla E(U),~~ 0 < t <\infty, \\
	\displaystyle \\
	\displaystyle U(0)=U_0(x),
\end{array}	
\right .
\end{equation}
with the constraint $U(t) \in \mathcal{M}^N$. The detailed process of gradient
flow is : starting from the initial value $U_0$, the solution goes along with
some curve $U=U(t) \in \mathcal{M}^N$ and finally, the critical point $U^*$
can be reached. Naturally, it is required that the value of the functional
decreases most fastly along the curve $U=U(t)$.

It should be mentioned that a non-trivial challenge 
for solving the above problem is to guarantee that the 
solution $U(t)$ always keeps on the $\mathcal{M}^N$ during the 
whole simulation. Fortunately, the extended
gradient flow model proposed in \cite{dai2020gradient} has been 
introduced to resolve this issue.

Firstly, the standard gradient of $E(U)$ is expressed as
\begin{equation}
	\nabla E(U)=(E_{u_1},E_{u_2},\cdots,E_{u_N}) \in (H^{-1}(\Omega))^N,
\end{equation}
where $E_{u_i} \in H^{-1}(\Omega)$ is defined by
\begin{equation}
	E_{u_i}=\frac{\delta E(U)}{\delta u_i}.
\end{equation}
Define the gradient flow on the Grassmann manifold $G^N$ of $E(U)$ \cite{edelman1998geometry},
\begin{equation}
	\tilde{\nabla}_G E(U)=\nabla E(U)-U \langle \nabla E(U),U \rangle^T,
	~~ \forall ~U \in \mathcal{M}^N.
\end{equation}
To preserve the orthonormality property, the domain of $\tilde{\nabla}_G E(U)$
from $\mathcal{M}^N$ to $(H_0^1(\Omega))^N$ need to be extended, i.e.,
\begin{equation}
	\nabla_G E(U)= \nabla E(U) \langle U^T U \rangle-U \langle \nabla E(U),U \rangle^T,~~ \forall ~ U \in (H_0^1(\Omega))^N.
\end{equation}
Based on the extended gradient $\nabla_G$, the gradient flow model based on
DFT can be written as
\begin{equation}\label{GF}
\left \{
\begin{array}{l}
	\displaystyle \frac{dU}{dt}=-\nabla_G E(U),~~ 0 < t <\infty, \\
	\displaystyle \\
	\displaystyle U(0)=U_0(x),
\end{array}	
\right .
\end{equation}
where the initial condition $U_0(x) \in \mathcal{M}^N$.
Meanwhile, the inner product of $U$ takes the following expression
\begin{equation}
	(U,\hat{U})=tr(\langle U^T \hat{U}\rangle),
\end{equation}
with the associated norm
\begin{equation}
	\lvert \lvert \lvert  U \rvert \rvert \rvert=(U,U)^{1/2},~~~~~~
	\lvert \lvert \lvert  U \rvert \rvert \rvert \in (H_0^1(\Omega))^N.
\end{equation}
The set of ground state for the Kohn--Sham gradient 
flow system are defined as
\begin{equation}
	\Theta=\Big\{ \Phi\in \mathcal{M}^N:
	E(\Phi)=\min_{U \in \mathcal{M}^N}E(U) ~\mbox{and} ~ \Phi~ \mbox{solves} ~(\ref{GF})  \Big\},
\end{equation}

\subsection{The fully discretized scheme}

In this subsection, the temporal discretization of the gradient flow model
is introduced firstly, including some properties of the temporal scheme,
and the discription of the spatial discretization is given following.
\subsubsection{Temporal discretization}

Following the model proposed in \cite{dai2020gradient}, a fully discretized
scheme will be given in this section, in which a linear numerical scheme
will be adopted for the temporal discretization. it
is worth mentioning that with this linear solver, the properties such as
orthonormality preserving among the wave functions and the decay of the
total energy hold well.

First of all, the gradient flow model is given as follows:
\begin{equation}
\left \{
\begin{array}{l}
	\displaystyle \frac{dU}{dt}=-\nabla_G E(U),~~ 0 < t <\infty, \\
	\displaystyle \\
	\displaystyle U(0)=U_0(x),
\end{array}	
\right .
\end{equation}
where the initial condition $U_0(x) \in \mathcal{M}^N$. The linear
temporal scheme given in \cite{hu2022} is proposed as follows.

Let $\{t_n: n=1,2,3,\cdots\} \subset [0,\infty)$ be discrete points
such that
\begin{equation}
	0=t_0 < t_1 < t_2 < \cdots < t_n < \cdots.
\end{equation}
Denote
\begin{equation}
	U^{n+\frac{1}{2}}=\frac{U^n+U^{n+1}}{2},
\end{equation}
then the linear temporal scheme is expressed as
\begin{equation}\label{linear:sch}
	\frac{U^{n+1}-U^n}{\Delta t_n}=-\mathcal{A}_{U^n}U^{n+\frac{1}{2}},
\end{equation}
where $\mathcal{A}_{U^n}$ is given by $\mathcal{A}_{U^n}=\{\nabla E(U^n),U^n\}
:=\nabla E(U^n)(U^n)^T-U^n \nabla E(U^n)^T$.

It has been verified in \cite{hu2022} that the orthonomalization among
the wavefunctions can be preserved well under this linear scheme, which is
guaranteed by the following proposition.

\begin{proposition}
	If $U^0 \in \mathcal{M}^N$, and $U^n$ is produced from (\ref{linear:sch})
then we obtain $U^n \in \mathcal{M}^N$.
\end{proposition}

\begin{proposition}\label{monotonic}
	If $U^0 \in \mathcal{M}^N$, and $U^n$ is obtained from (\ref{linear:sch}).
Given $\Delta t_n$ is extremely	small, and assume $\nabla E(U)$ is Lipschitz
continuous in the $\mathcal{M}^N$, i.e.,
\begin{equation}
	\lvert \lvert \lvert \nabla E(U)-\nabla E(\Phi)\rvert \rvert \rvert
	\leq L \lvert \lvert \lvert U-\Phi \rvert \rvert \rvert,
	~~\forall ~ U,\Phi \in \mathcal{M}^N,
\end{equation}
then the sequence $\{U^n\}$ satisfies
\begin{equation}
	E(U^{n+1}) \leq E(U^n).
\end{equation}
\end{proposition}
This proposition is related to the decay behavior of the total energy,
which is also an important property to guarantee the convergence of the
solution towards the ground state of a given system.

\subsubsection{Spatial discretization}
In order to propose a fully discretized scheme, the classical finite
element method is employed fort the spatial discretization, which is
introduced in this subsection.

To derive the finite element discretization for the gradient flow system,
we firstly introduce the following notations for description. First of all,
$\Omega \subset \mathbb{R}^3$ is used to denote the computational domain
, and $\partial \Omega $ is its boundary. For this domain $\Omega$, we
have a tetrahedron mesh $\mathcal{T}$ which completely covers the domain
$\Omega$. The mesh $\mathcal{T}$ consists of a set of nonoverlapped
tetrahedron elements, i.e., $\mathcal{T}=\{ T_k \}_{k=1}^{N_{tet}}$,
where $N_{tet}$ is the total number of the tetrahedron elements in the
mesh $\mathcal{T}$. Then following the definition introduced by Ciarlet,
on tetrahedron elements $T_k, k=1,\cdots, N_{tet}$, we define the finite
element $(T_k,P_r,\mathbb{N})$, where $P_r$ is the set of all polynomial
with degree no larger than r in three variables, and $\mathcal{N}$ is the
set of nodal variables. With the above notations,
the standard Sobolev space in a given bounded domain
$\Omega \subset \mathbb{R}^3$ is defined by $H^1:=W_2^1(\Omega)$, 
and the standard Lagrange finite element space takes the form
\begin{equation}
S_h:=\{ \phi_h \in C(\Omega)~: \phi_h|_{\tau} \in P_{\tau}, ~\forall \tau \in \mathcal{T}_h,~ \phi_h|_{\partial \Omega}=0 \}.
\end{equation}
The set of finite dimensional Kohn--Sham gradient flow 
ground state solutions is denoted as follows,
\begin{equation}
	\Theta_h=\Big\{ \Phi_h\in S_h  \cap \mathcal{M}^N:
	E(\Phi_h)=\min_{U \in \mathcal{M}^N\cap V_h}E(U) ~\mbox{and} ~\Phi_h~ \mbox{solves} ~ \mbox{the weak form of (\ref{GF})}  \Big\}.
\end{equation}
Together with the temporal discretized scheme, the full-discrete problem
of (\ref{linear:sch}) can be formulated as: Find $U_h^{n+1} \in (S_h)^N$ 
such that
\begin{equation}\label{fulldiscrete}
(U_h^{n+1},\Phi_h)-\frac{\Delta t_n}{2}  a_n(U_h^{n+1},\Phi_h)=	
\frac{\Delta t_n}{2}a_n(U_h^n,\Phi_h)+(U_h^n,\Phi_h),~~ 
\forall ~\Phi_h \in (S_h)^N,
\end{equation}
where $U_h^n$ is given and
\begin{equation}
	 a_n(U,\Phi)=(\nabla E(U_h^n) \langle U_h^n, U \rangle, \Phi )-(U_h^n
	 \langle \nabla E(U_h^n), U \rangle, \Phi ).
\end{equation}
Additionally, the following assumptions are proposed in 
\cite{chen2014adaptive}, which would be used in the analysis of the 
spatial convergence in this paper.

Given an initial mesh $\mathcal{T}_0$, based on the refined strategy of
the adaptive finite element method, we can get a series of finite element
subspaces and its corresponding meshes defined as follows:
\begin{equation}
	S_0^{h_0}(\Omega) \subseteq S_0^{h_1}(\Omega) \subseteq \cdots
	\subseteq S_0^{h_k}(\Omega) \subseteq S_0^{h_{k+1}}(\Omega)
	\subseteq \cdots \subseteq S_{\infty}(\Omega) \subseteq H^1_0(\Omega),
\end{equation}
where $S_{\infty}(\Omega)=\overline{\cup_{k=0}^{\infty}S_0^{h_k}(\Omega)}^{H_0^1(\Omega)}$, with a sequence of meshes
$\mathcal{T}_1, \mathcal{T}_2, \cdots$.

For the partition $\mathcal{T}_k$, it can be divided into two subsets
$\mathcal{T}_k^+$ and $\mathcal{T}_k^0$, such that
\begin{equation}
	\mathcal{T}_k^+=\{ \tau \in \mathcal{T}_k: \tau \in \mathcal{T}_l, ~~\forall ~l \geq k\} ~~\mbox{and}~~ \mathcal{T}_k^0=\mathcal{T}_k \backslash 
	\mathcal{T}_k^+,
\end{equation}
where the set $\mathcal{T}_k^+$ contains the elements which are not refined
anymore, and the elements in set $\mathcal{T}_k^0$ will eventually be refined.
So the computational domain $\Omega$ can be denoted in a similar way,
\begin{equation}
	\Omega_k^+=\mathop{\cup}\limits_{\tau \in \mathcal{T}_k^+}\omega_k(\tau)~~\mbox{and}~~
	\Omega_k^0=\mathop{\cup}\limits_{\tau \in \mathcal{T}_k^0}\omega_k(\tau).
\end{equation}
Suppose that $V_{\infty}$ is a Hilbert space which is denoted by
$V_{\infty}=(S_{\infty}(\Omega))^N$ with the inner product
inherited from $\mathcal{H}$ and assume for any
$\Psi_{\infty} \in V_{\infty}$, there holds
\begin{equation}
	\lim_{k \rightarrow \infty} \inf_{U_k \in V_{h_k}} \lVert U_k-U_{\infty}\rVert_{1,\Omega}=0,
	\end{equation}
Using a direct calculation, it follows
\begin{equation}
	\inf_{\tilde{U}_k\in V_{h_k}\cap \mathcal{O}^N} \lVert \tilde{U}_k-
	U_{\infty}\rVert_{1,\Omega} \lesssim \inf_{U_k \in V_{h_{k}}}
	\lVert U_k-U_{\infty}\rVert_{1,\Omega}, ~~\forall~ U_{\infty}
	\in V_{\infty}\cap \mathcal{O}^N, ~~\forall~ k \in \mathbb{N}_0,
\end{equation}
thus, there holds
\begin{equation}\label{appro}
	\lim_{k \rightarrow \infty} \inf_{\tilde{U}_k\in V_{h_k}\cap \mathcal{O}^N}
	\lVert \tilde{U}_k-U_{\infty}\rVert_{1,\Omega}=0,~~\forall~ U_{\infty}
	\in V_{\infty}\cap \mathcal{O}^N.
\end{equation}

\subsection{Algorithm}

In this subsection, focusing on the proposed method for the gradient flow
model, an adaption scheme is given in order to improve the efficiency of 
implementation.
The algorithm below shows how the adaptive strategies work on
our numerical scheme, and the corresponding experiment results will
be shown in Section 5.

\begin{algorithm}[H]

 Set $k=0, n=0$, given an adaptive initial grid $\mathcal{T}_0$ and initial
 value $U_0^0 \in (S_0)^N \cap \mathcal{M}^N$. Choose a positive integer
\textit{maxrefine} based on the computer memory.

 Choose $\Delta t=\min (h_{\tau_i})^2$, $\tau_i$ is the $i$th element of
 current mesh $\mathcal{T}_k$.

 Find $U_k^{n+1} \in (S_k)^N$ such that
 \begin{equation}
 	(U_k^{n+1},\Phi_k)-\frac{\Delta t}{2}a_n(U_k^{n+1},\Phi_k)=\frac{\Delta t}{2}a_n(U_k^{n},\Phi_k)+(U_k^n,\Phi_k),~~ \forall ~\Phi_k \in (S_k)^N.
 \end{equation}

If $E(U_k^{n+1} \geq E(U_k^n))$, set $\Delta t= \frac{\Delta t}{2}$ and
go back to line 3.

If $(mod(n,200)==0)$, set $\Delta t =2\Delta t$.

If $(\frac{E(U_k^n)-E(U_k^{n+1})}{\Delta t} < \epsilon )$,
$U_k^{end}=U_k^{n+1}$, goto line 7 ; else set $n=n+1$ and go back to line 3.

If $(K < maxrefine)$ goto line 8 ; else stop the whole simulation.

Estimate. Compute an a \emph{posteriori} error estimator $\eta_{k,\tau}$
on $\mathcal{T}_k$ \;

Mark. Select the \textbf{minimal} elements set $M_k \subset \mathcal{T}_k$
such that
\begin{equation}
	\sum_{\tau \in M_k} \eta_{k,\tau}^2 \geq \theta
	\sum_{\tau \in M_k} \eta_{k,\tau}^2, ~~ for ~some ~ \theta \in (0,1).
\end{equation}

Refine. Bisect the marked element $M_k$ and get the mesh $\mathcal{T}_{k+1}$\;

Set $U_{k+1}^0=U_k^{end}$, and orthogonalize $U_{k+1}^n$, $n=0, k=k+1$,
go back to line 2.
\caption{Adaptive computing}	
 \label{al:1}
\end{algorithm}

An $h$-adaptive finite element method is employed for efficiency
in this paper. A classical process of using the adaptive mesh methods
consists of the following steps,
\begin{equation}
	\cdots \mbox{solve} \cdots \mbox{estimate} \cdots \mbox{mark} \cdots \mbox{refine} \cdots
\end{equation}
The above whole process means that the gradient flow model is solved firstly
on the current finite element space, then the distribution of the numerical
error on the current mesh is estimated, next, by the means of the
\emph{a posteriori} error estimator, the tetrahedron elements of the current
mesh are marked. Finally, a new finite element space is built on the new
mesh by locally refining or coarsening the current mesh, and the new solutions
are obtained by the interpolation. In this paper, a recovery type
\emph{a posteriori} error estimation is used, this ostensibly rather crude
approach can result in astonishingly good estimates for the true error, which
is given as follows,
\begin{equation}\label{errorindicator}
	\eta_{h,\tau}=h_{\tau} \int_{\tau} \lvert \nabla R (\nabla \rho)\rvert^2 dx,
\end{equation}
in which $R$ stands for the recovery operator \cite{xu2004analysis}.

Moreover, besides the adaptive mesh method, the adaptive time step
strategy is also employed in our paper. Theoretically, the exact
solution of the gradient flow problem can be reached when time $t$
tends to infinity, which means a quite large size of time step is
preferred to accelerate the simulation. The details of the temporal
adaption is depicted as follows: firstly, we set $\Delta t=\min h_{\tau_i}^2$,
where $\tau_i$ is the $i$-th element of the current mesh, then if $\Delta t$
is too large, i.e., $E(U^n)-E(U^{n+1})<0$, we adjust the time step to
$\Delta t= \Delta t/2$. Furthermore, the time step $\Delta t$ would be
enlarged to $2\Delta t$ every 200 steps in order to accelerating the
simulation.

Based on the above linearized numerical scheme, we try
to deliver a complete convergence analysis on it for the all-electron
gradient flow based Kohn--Sham model both temporally and spatially. 
The details are demonstrated in the next section.

\section{Convergence analysis of the numerical scheme}

The gradient flow method has been an attractive approach obtaining 
the ground state of the KSDFT. By an in-depth analysis of the 
model problem, the convergence of our numerical scheme for the gradient 
flow based KSDFT model will be analyzed in this section.
Actually, we want to show the convergence of our numerical scheme
theoretically under some proper assumptions, i.e.,
\begin{equation}\label{convergence}
	E(U_k^n) \rightarrow E(U),~~ k,n \rightarrow \infty,
\end{equation}
where $U$ is the exact solution of the model problem. The proof of
(\ref{convergence}) can be divided into two steps:
\begin{eqnarray}
	&& E(U_k^n) \rightarrow E(U_k^*), ~~~n \rightarrow \infty, \\
	&& E(U_k^*) \rightarrow E(U),~~~~ k \rightarrow \infty.
\end{eqnarray}
The first step shows the temporal convergence, and the second step
indicates the spatial convergence. The details of the proof will be
demonstrated in the following two subsections.

\subsection{Temporal convergence}

The convergence of the linearized temporal scheme 
has been investigated in this subsection. Before given the convergence 
analysis, the Lyapunov stability should be introduced firstly, which 
lays a solid foundation for the convergent analysis of the linear solver.

Various type of stability may be discussed for
the solutions of differential equations describing dynamical systems,
and the most important type is concerning the stability of solutions
near to an equilibrium point, which is the so-called Lyapunov stability.
Additionally, the equilibrium of a system is said to
be asymptotically stable if it is Lyapunov stable, and there exists
$\delta >0$ such that if $\lVert U(0)- U^*\rVert < \delta$, then
as is demonstrated in \cite{dai2020gradient} for the gradient flow model,
\begin{equation}\label{asymstab}
  \left \{
  \begin{array}{l}
  	\displaystyle \lim_{t \rightarrow \infty}  \lvert \lvert \lvert 
  	\nabla_G E \big(U(t)\big ) \rvert \rvert \rvert=0, \\
  	\displaystyle \\
  	\displaystyle \lim_{t \rightarrow \infty} E\big(U(t)\big)=E(U^*).
  \end{array}
    \right .
\end{equation}
The above relation (\ref{asymstab}) describes 
the behaviour of asymptotic stability which is more strongly 
than the Lyapunov stability. Concerning the asymptotic stability
for the Kohn--Sham system, as $t \rightarrow \infty$, the energy 
$E \big(U(t)\big )$ tends to $E(U^*)$, which hints the equilibrium point 
$U^*$ is a local minimizer of the Kohn--Sham energy functional. 

Additionally, the Lemma \ref{le:1} come up in 
\cite{dai2020gradient} will be employed in the demonstration, which 
depicts the Lipschitz continuity of the gradient flow,
\begin{lemma}\label{le:1}
There holds
\begin{equation}
	\lvert \lvert \lvert \nabla_G E(U_i)-\nabla_G E(U_j) \rvert \rvert \rvert
	\leq L_1 \lvert \lvert \lvert E(U_i)- E(U_j) \rvert \rvert \rvert, ~~
	\forall ~ U_i,U_j \in B(U^*,\mbox{max}\{\delta_a,\delta_b\})
\end{equation}	
where $L_1=2\alpha(2\lvert \lvert \lvert \nabla E(U^*) \rvert \rvert \rvert)
+2L max\{\delta_a,\delta_b\}+\alpha L$. Moreover, there exists an upper bound
$\delta s$ of s such that
\begin{equation}
	E(U)-E(g(U,s)) \geq \frac{s}{4N} \Big\lvert \Big\lvert \Big\lvert
	\nabla_G E \Big(\frac{g(U,s)+U}{2} \Big) \Big\rvert \Big\rvert \Big\rvert,
	~~\forall~ U \in B(U^*,\delta_a) \cap \mathcal{M}^N,~~\forall ~ s\in
	[0,\delta_s],
\end{equation}
where $\delta_a,\delta_b$ is related to the function $g(U,s)$.
\end{lemma}
$B(U^*,\mbox{max}\{\delta_a,\delta_b\})$, and $B(U^*,\delta_a)$ in Lemma
\ref{le:1} are some type of closed $\delta-$neighbor of U.
The Lemma \ref{le:1} verifies the local Lipschitz continuous in the
neighborhood of a local minimizer $U^* \subset (S_h)^N \cap \mathcal{M}^N$.

Based on the above description of the asymptotic
stability for the gradient flow system and in order to go further,
we try to demonstrate that the asymptotic stability also suit 
for the linear solver employed in this paper. For simplicity, hereafter 
we denote the wavefunction $U_k^n$ as $U^n$ in the following proof. The
demonstration is shown as follows.
\begin{theorem}
If $U_k^0 \in B(U_k^*,\delta)$, and $\Delta t$ is small enough, then
the sequence $U_k^n$ produced by algorithm \ref{al:1} satisfy
\begin{eqnarray}
   &\lim\limits_{n \rightarrow \infty} \lvert \lvert \lvert 
   \nabla_G E(U_k^n) \rvert\rvert \rvert  =  0, \\
   &\lim \limits_{n \rightarrow \infty} E(U_k^n)  =  E(U_k^*).
\end{eqnarray}
\end{theorem}
\begin{proof}
Before approaching the proof, it should be mentioned that all the 
$\delta (\delta_T,\delta_e, \cdots
.)$
involved in the analysis are assumed to be small enough concerned with
the setting of the Lyapunov stability.

Using the Proposition \ref{monotonic} mentioned before, we can easily 
find that the sequence of the approximations of total energy 
$\{E[U^n]\}_{n=0}^{\infty}$ decreases monotonically, moreover, 
$\{E[U^n]\}_{n=0}^{\infty}$ is assumed to be bounded below, thus, 
$\lim \limits_{n \rightarrow \infty} E([U^n])$ exists.

Employing the lower bound for the gap between $E(U^n)$ and 
$E(U^{n+1})$, which is come up in Lemma \ref{le:1},
\begin{eqnarray}
	\sum_{n=0}^{\infty} \frac{\Delta t_n}{4N} \lvert \lvert \lvert
	\nabla_G E(U^{n+1/2})\rvert \rvert \rvert^2 & \leq & \lim_{n \rightarrow \infty} \Big(E(U^0)- E(U^1)+E(U^1)-E(U^2)+ \cdots \nonumber \\ 
	   +E(U^n) -E(U^{n+1})\Big) \nonumber \\
	& \leq & E(U^0)- \lim_{n \rightarrow \infty} E(U^n) < \infty,
\end{eqnarray}
as a result of $\sum \limits_{n=1}^{\infty} \Delta t_n \rightarrow \infty$,
the gradient flow term
$\lvert \lvert \lvert \nabla_G E(U^{n+1/2})\rvert \rvert \rvert$ yields
\begin{equation}
	\lim_{n \rightarrow \infty} \inf\lvert \lvert \lvert \nabla_G E(U^{n+1/2})\rvert \rvert \rvert=0.
\end{equation}
Based on the definition of limit inferior, assume
that $\{U^{n_{k+1/2}}\}_{k=0}^{\infty}$ is a subsequence of
$\{U^{n+1/2}\}_{n=1}^{\infty}$, then there holds
\begin{equation}
	\lim_{k \rightarrow \infty} \lvert \lvert \lvert U^{n_{k+1}}-U^{n_k}
	\rvert \rvert \rvert \leq \delta_T \lim_{k \rightarrow \infty} \lvert \lvert \lvert \nabla_G E(U^{n_{k+1/2}})\rvert \rvert \rvert=0,
	\end{equation}
which is obtained based on our linearized temporal scheme.

It has been mentioned that
	$\hat{S}\equiv \Big\{ U \in (S_h)^N: [U] \in B([U^*],\delta_e) \cap
	\mathcal{L}_e \Big \}$ \cite{dai2020gradient}
is a compact set, and every sequence in a compact set has a convergent
subsequence, thus without loss of generality, the subsequence
 $\{U^{n_k}\}_{k=1}^{\infty}$ satisfies
 \begin{equation}
 	\lim_{k \rightarrow \infty} U^{n_k}=\tilde{U},
 	~~~~ \mbox{for some}~\tilde{U} \in \hat{S}.
 \end{equation}
Thus, motivated by the Lipchitz continuity of $\nabla_G$, 
the formulation of the linear solver, as well as the monotonically 
decreasing property of total energy, we can obtain
\begin{equation}
	\nabla_G E(\tilde{U})=0,
\end{equation}
which hints
\begin{equation}	
	\lim_{n \rightarrow \infty}\inf \lvert \lvert \lvert 
	\nabla_G E(U^{n})\rvert \rvert \rvert =0.
\end{equation}
Since the sequence $\big\{\nabla_G E(U^n)\big\}$ 
is Lipschitz continuous, which is defined in a small neighborhood of 
a local minimizer $U^*$, i.e., $B(U^*,max\{\delta_a,\delta_b\})$, 
where $\delta_a$, $\delta_b$ are assumed to be small enough,
hence we have
\begin{equation}
	\lim_{n \rightarrow \infty}  \lvert \lvert \lvert \nabla_G E(U^n)\rvert \rvert \rvert=0.
\end{equation}
Recalling the Lemma \ref{le:2} proposed in \cite{dai2020gradient}, 
\begin{lemma}\label{le:2}
	There holds
	\begin{equation}
		\hat{g}\Big( B([U^*],\delta_e) \bigcap \mathcal{L}_{E_e}\times [0,\delta_T] \Big) \subset B([U^*],\delta_e)\bigcap
		\mathcal{L}_{E_e},
	\end{equation}
for some $\delta_e >0$, $E_e \in \mathbb{R}$, $\delta_T \in[0,\delta_s]$,
where $\delta_s$ is defined as in Lemma	\ref{le:1}.
\end{lemma}
It can tell us the following relation,
\begin{equation}\label{relation}
	[\tilde{U}] \in B([U]^*,\delta_e)\cap \mathcal{L}_e
 \subset B([U]^*,\delta_c),
 \end{equation}
which is an important step in this proof. However, this relation can not be used
directly owing to our linearized temporal scheme. In order to realize it,
we have the following analysis.

Firstly, we concerned with the notation $g(t)=tU^{n-1}+(1-t)U^n$, 
suppose $E(g(t))$ is differential in $(0,1)$, which implies that 
there exists a number $\xi \in (0,1)$ such that
\begin{eqnarray}\label{Eerror}
   E(U^{n+1})-E(U^n)&=&E(g(1))-E(g(0))\nonumber\\
    &=& \frac{d(E(g(t)))}{dt}\lvert_{t=\xi} \cdot (1-0)\nonumber\\
    &=& tr \big\langle \nabla E (g(\xi))^T \cdot g'(\xi)\big \rangle\\
    &=& tr\big\langle \nabla E (g(\xi))^T \cdot (U^{n+1}-U^n)\big\rangle \nonumber\\
    &=& -\Delta t_n tr\big\langle \nabla E (g(\xi))^T \cdot A_{U^n}U^{n+\frac{1}{2}}\big\rangle. \nonumber
\end{eqnarray}
From the temporal discretization, the following expression can be reached,
\begin{equation}
    U^{n+\frac{1}{2}}=(I+\frac{\Delta t_n}{2}A_{U^n})^{-1}U^n.
\end{equation}
Take this expression into (\ref{Eerror}), we can get
\begin{eqnarray}
  && tr\big\langle \nabla E (g(\xi))^T \cdot A_{U^n}U^{n+\frac{1}{2}}\big\rangle \nonumber\\
  &=&tr\big\langle \nabla E (g(\xi))^T \cdot A_{U^n}\cdot (I+\frac{\Delta t_n}{2}A_{U^n})^{-1}U^n \big\rangle \\
  &=& tr\big\langle \nabla E (g(\xi))^T \cdot A_{U^n}U^n \big\rangle-\frac{\Delta t_n}{2}tr\big\langle \nabla E (g(\xi))^T \cdot A_{U^n}A_{U^n}U^n \big\rangle +o(\Delta t_n). \nonumber
\end{eqnarray}
Since the second term is higher order than the first term, we 
only need to estimate the first term 
$tr\big\langle \nabla E (g(\xi))^T \cdot A_{U^n}U^n \big\rangle$.

Now that the temporal scheme is different from \cite{dai2020gradient}, it
follows that the inequality given below need to be verified,
\begin{equation}
  \lvert \lvert\lvert A_{U^n}U^{n+\frac{1}{2}}-A_{U^*}U^* \rvert\rvert\rvert  
  \leq L_1\lvert\lvert\lvert U^n-U^* \rvert\rvert\rvert,
\end{equation}
which is concerned with the Lemma 4.9 in \cite{dai2020gradient}.

We begin with the following expression,
\begin{eqnarray}\label{transform}
  \lvert\lvert\lvert A_{U^n}U^{n+\frac{1}{2}}\rvert\rvert\lvert^2 &=&
  tr \big\langle  (A_{U^n}U^{n+\frac{1}{2}})^T \cdot A_{U^n}U^{n+\frac{1}{2}}\big\rangle \nonumber\\
  &=& -tr \big\langle  (A_{U^n}U^{n+\frac{1}{2}}) \cdot A_{U^n}U^{n+\frac{1}{2}}\big\rangle \nonumber\\
  &=& -tr \big\langle  (A_{U^n}\cdot(I+\frac{\Delta t_n}{2}A_{U^n})^{-1}U^n) \cdot (A_{U^n}\cdot(I+\frac{\Delta t_n}{2}A_{U^n})^{-1}U^n)\big\rangle \nonumber\\
  &=& -tr \big\langle (A_{U^n}\cdot U^n)^T \cdot (A_{U^n}\cdot U^n)\big\rangle + \mbox{high order terms}\nonumber\\
  &=& \lvert\lvert\lvert A_{U^n}U^n\rvert\rvert\rvert^2+\mbox{high order terms}. 
\end{eqnarray}
Since the following relations
\begin{eqnarray}
  \lVert \nabla_G E(U_i)- \nabla_G E(U_j)\rVert &\leq& L_1 \lVert U_i-U_j\rVert,  \\
   \nabla_G E(U)=A_U U,
\end{eqnarray}
hold, thus the Lipschitz continuity is also satisfied for $A_U$, i.e.,
\begin{equation}
 \lvert\lvert\lvert A_{U^n}U^{n+\frac{1}{2}}\rvert\rvert\rvert  \leq 
 \lvert\lvert\lvert A_{U^n}U^n\rvert\rvert\rvert,
\end{equation}
which hints
\begin{eqnarray}
  \lvert\lvert\lvert A_{U^n}U^{n+\frac{1}{2}}-A_{U^*}U^*\rvert\rvert\rvert 
  &\leq & \lvert\lvert\lvert  A_{U^n}U^n-A_{U^*}U^*\rvert\rvert\rvert  \nonumber\\
  & \leq & L_1 \lvert\lvert\lvert U^n-U^*\rvert\rvert\rvert \\
  & \leq & L_1 \mbox{max}\{\delta a, \delta b\}.\nonumber
\end{eqnarray}
Additionally, the high order terms involved in (\ref{transform}) 
should be estimated, which is organized as:
\begin{equation}
    -\Delta t_n tr \big\langle A_{U^n}U^n\big\rangle \cdot
    tr \big\langle A_{U^n}A_{U^n}U^n\big\rangle +o(\Delta t_n).
\end{equation}
The assumption of the time step in the convergence of
temporal scheme for gradient flow model is given as
\begin{equation}
    \sup\{ \Delta t_n: n \in \mathbb{N} \} \leq \delta_T.
\end{equation}
Concerned with the assumption that $\delta_T$ involved in this analysis 
should be  small enough.  It follows that the high order terms mentioned
above can be organized as follows:
\begin{equation}
	-\Delta t_n tr \big\langle A_{U^n}U^n\big\rangle \cdot tr \big\langle
	A_{U^n}A_{U^n}U^n\big\rangle+o(\Delta t_n),
\end{equation}
Based on the assumption that
\begin{equation}
	\sup \{\Delta t_n: n \in \mathbb{N}\} \leq \delta_{T},
\end{equation}
we can assume $\lim \limits_{n \rightarrow \infty} \Delta t_n=0$.
In addition, the boundedness of $tr \big\langle A_{U^n}U^n\big\rangle$
and $tr \big\langle A_{U^n}A_{U^n}U^n\big\rangle$ can be easily 
reached by employing the local Lipschitz continuity of $\nabla_G E$.

Based on the above analysis, the high order terms 
can be neglected asymptotically concerning with the definition of 
the closed $\delta-$neighborhood $B([U^*],\delta)$. In such a 
case, we can employ the Lemma \ref{le:2}, and obtain the 
relation (\ref{relation}).

Consequently, it is trivial to get $[\tilde{U}]=[U^*]$, and
\begin{equation}
	\lim_{n\rightarrow \infty} E(U^n)= \lim_{k \rightarrow \infty} E(U^{n_k})
	= E(U^*),
\end{equation}
which completes the proof.
\end{proof}

\subsection{Spatial convergence}

Motivated by the work \cite{chen2014adaptive}, in which
Chen and co-authors have proved the convergence of AFE approximations
under the pseudopotential framework. In this paper, the convergence of
the AFE approximations is studied for the all-electron Kohn--Sham
gradient flow model. Several assumptions and Lemmas will be introduced
firstly before demonstrating the convergence of the AFE approximations.

It is noted that $V_{xc}=d \varepsilon_{xc}/d \rho$, where
$\varepsilon_{xc}$ is the exchange-correlation energy per unit volume.
The following assumptions \cite{chen2013numerical} give the
boundedness of $V_{xc}$ directly,

\begin{itemize}
\item[$A_1$]:
  $\lvert\varepsilon_{xc}'(t) \lvert+\lvert t\varepsilon_{xc}''(t)
  \lvert \in \mathbb{P}(p_1,(c_1,c_2)) $ for some $p_1\in [0,2]$,
where $\mathbb{P}$ is a functional space defined as
   \begin{equation}
     \begin{array}{l}
     \mathbb{P}(p_1,(c_1,c_2)) =\\
       \\
       \hfill \quad\quad\{ f:\exists \ a_1,a_2 \in \mathcal{R} \ \mbox{such that} \ c_1t^{p_1}+a_1 \leq f(t)\leq c_2t^{\textcolor{black}{p_1}}+a_2, \forall t\geq 0\}.
\end{array}
\end{equation}
\item[$A_2$]: There exists a constant $\alpha \in (0,1]$ such that
  \begin{equation}
    \lvert\varepsilon_{xc}''(t) \lvert+\lvert t\varepsilon_{xc}'''(t) \lvert \lesssim  1+t^{\alpha-1}  \quad \forall t>0.
  \end{equation}
\end{itemize}

Furthermore, the following lemma which has been demonstrated
in our previous work \cite{shen2019asymptotics} gives the boundedness
of the external energy.
\begin{lemma}
Let $R_k$ be the location of the k-th nuclei with the nuclear charge $Z_k$,
$u_i$ is the i-th wavefunction, then there exists a constant c such that
\begin{equation}
	\int_{\mathbb{R}^3}\frac{Z_k}{\lvert \bm{r}-R_k\rvert}\lvert u_i\rvert^2
	d \bm{r} \leq c \int_{\mathbb{R}^3} \lvert \bm{r}-R_k\rvert
	\lvert \nabla u_i\rvert^2 d \bm{r}.
\end{equation}
\end{lemma}
By employing the above lemmas and some analysis, we can get the 
boundedness of the total energy of the all-electron Kohn--Sham system,
and also get an \emph{a posteriori} error estimator for the total 
energy, i.e., $\eta_{loc}$ \cite{shen2019asymptotics}, which takes
the following expression,
\begin{equation}
	\eta_{loc}=\sum_{i=1}^N \lVert u_i^h-\tilde{u}_{h,i}^H
	\rVert_{1,\mathcal{T}_k^h}^2+\lVert \tilde{\phi}_h^H-\phi^h\rVert_{1,\mathcal{T}_k^h}^2+\frac{1}{8 \pi}
	\lvert \tilde{\phi}_h^H-\phi^h \rvert_{1,\mathcal{T}_k^h}^2
\end{equation}
Moreover, based on some classical results proposed in finite element 
methods \cite{ciarlet2002finite}, an upper bound for $\eta_{loc}$ 
has been obtained, which yields,
\begin{equation}
 \eta_{loc} \leq C (h_K)^{2k} \lvert u^h\rvert_{k+1,\mathcal{T}_K^h}^2,
\end{equation}
where $k$ stands for the order of the approximate polynomials. Without
loss of generality, we take $k=1$ in the following analysis. 
By using the Sobolev imbedding theorem, we can reach
\begin{eqnarray}\label{upperbound}
	\eta_{loc} &\leq & C (h_K)^{2} \lvert u^h\rvert_{2,\mathcal{T}_K^h}^2
	\nonumber\\
	           & \leq & \tilde{C} \lVert u^h\rVert^2_{2,6,\omega_h(\tau)},
\end{eqnarray}
where $\omega_h(\tau)$ stands for an element defined in the spatial 
partition mentioned before. To unify the expression, we denote 
$U^h$ by $U_k$ in the following content.

Before proceeding further, we need to verify the following Lemma proposed
in \cite{chen2014adaptive} holds for the all-electron Kohn--Sham gradient
flow model as well.
\begin{lemma}\label{le:4}
	Let $\{\Phi_k\}_{k \in \mathbb{N}}$ be the sequence produced by
	the Algorithm \ref{al:1}. If Assumption \textbf{A1} is satisfied, then
	\begin{equation}
		\lim_{k \rightarrow \infty} \max_{\tau \in \mathcal{M}_{k}}
		\eta_{loc}(U_{k},\tau)=0,
	\end{equation}
\begin{equation}
	\lim_{k \rightarrow \infty} \langle \bm{R}_k(U_k),\Gamma\rangle=0,~~
	\forall ~ \Gamma \in \mathcal{H}.
\end{equation}
\end{lemma}
\begin{proof}
According to the construction of the finite element subspaces 
and the definition of the dense set, we can find for any 
subsequences of $\{U_k\}$, there exists a convergent subsequence 
$\{U_{k_m}\}$ and $U_{\infty}$, $(\Lambda_{\infty},U_{\infty}) 
\in \Theta_{\infty}$ such that
\begin{equation}
	U_{k_m} \rightarrow U_{\infty}, ~~\mbox{in} ~\mathcal{H},
\end{equation}
for simplicity, hereafter we denote $U_{k_m}$ as $U_k$ in this proof.
Based on the (\ref{upperbound}), the following result can be reached,
\begin{eqnarray}\label{upp:2}
\eta_{loc}(U_k,\tau_k) \lesssim	\lVert U_k \rVert_{2,6, \omega_k(\tau_k)},
\end{eqnarray}
where $\mathcal{M}_k$ is constructed by using the Maximum Strategy in
the simulation, and $\tau_k \in \mathcal{M}_k$, such that
\begin{equation}
	\eta_{loc}(U_k,\tau_k)=\max_{\tau \in \mathcal{M}_k}
	\eta_{loc}(U_k,\tau).
\end{equation}
Based on the construction of the mesh size functions
$\{h_k\}_{k \in \mathbb{N}}$ in \cite{chen2014adaptive}, there holds
\begin{equation}
	\lvert \omega_k(\tau_k)\rvert \lesssim h_{\tau_k}^3 \lesssim
	\lVert h_k \chi_{\Omega_k^0}\rVert^3_{0,\infty,\Omega} \rightarrow 0,
	~~ \mbox{as}~ k \rightarrow \infty,
\end{equation}
which hints that the right-hand side of inequality (\ref{upp:2}) goes to
zero, which means that
\begin{equation}
	\lim_{k \rightarrow \infty} \max_{\tau \in \mathcal{M}_{k}}
		\eta(U_{k},\tau)=0.
\end{equation}
Moreover, a theorem given in \cite{dai2008convergence} has 
established a basic relationship between the error estimate of 
finite element eigenvalue approximations and the associated finite 
element boundary value solutions. Thus, concerned with the 
classical residual-based \emph{a posteriori} error estimate for
the boundary value problems, there exists an equivalent relationship
between the residuals, i,e., the element residual and the boundary 
residual and the error. According to the theorem 
\cite{dai2008convergence}, this equivalence holds for eigenvalue 
problems as well. Since the Kohn--Sham equation can be regarded as 
a generalized eigenvalue problem, there should be an equivalence 
between the residuals and the error of wavefunctions, 
moreover, recalling the formulation of the \emph{a posteriori} 
error indicator $\eta_{loc}$ mentioned above and employing the 
Poincar$\acute{e}$-Friedrichs inequality, there holds
\begin{equation}
	\lvert\langle \bm{R}_k(U_k),\Gamma\rangle \rvert \lesssim \sum_{\tau \in \mathcal{T}_k}\lVert U-U_k\rVert_{1,\omega_k(\tau)}
	\lVert \Gamma\rVert_{1,\omega_k(\tau)}
	\lesssim \sum_{\tau \in \mathcal{T}_k}\eta_{loc}(U_k,\tau)
	\lVert \Gamma\rVert_{1,\omega_k(\tau)},~~\forall ~ \Gamma \in \mathcal{H},
\end{equation}
which means
\begin{equation}
	\lim_{k \rightarrow \infty} \langle \bm{R}_k(U_k),\Gamma\rangle=0,~~
	\forall ~ \Gamma \in \mathcal{H},
\end{equation}
which completes the proof.
\end{proof}
Consequently, based on the above analysis, we are ready to state 
and prove our main result: showing that the following Theorem \ref{spacon} 
holds for the all-electron Kohn--Sham gradient flow model.
\begin{theorem}\label{spacon}
	Let $\{\Theta_k\}_{k \in \mathbb{N}}$ be the sequence generated by the
	Algorithm \ref{al:1}. If the initial mesh $\mathcal{T}_0$ is sufficiently
	fine and Assumption $\bm{A1}$ is satisfied, then
	\begin{equation}
		\lim_{k \rightarrow \infty}E_k=\min_{U \in \mathcal{O}^N }E(U).
\end{equation}
\end{theorem}


\begin{proof}
	
In the first step, we need to show that
\begin{equation}\label{firstresult}
	\lim_{k \rightarrow \infty} E_k=\min_{U \in V_{\infty}\cap \mathcal{O}^N}E(U).
\end{equation}
If we want to prove this limitation holds, we need to find a sequence
converges to $\Phi_{\infty} \in V_{\infty}$ and satisfying
\begin{equation}
	E(\Phi_{\infty})=\min_{U \in V_{\infty}\cap \mathcal{O}^N}E(U).
\end{equation}
Due to the boundedness of the finite dimensional approximations and the
Eberlein-Smulian theorem, for a given $\Phi_{\infty} \in V_{\infty}$, there
exists a weakly convergent subsequence $\{\Phi_{k_{mj}}\}_{j \in N_0}$
such that
\begin{equation}
	\Phi_{k_{mj}}\rightharpoonup \Phi_{\infty}~~ \mbox{in} ~\mathcal{H}.
\end{equation}
Since $H_0^1(\Omega)$ is compactly embedded in $L^p(\Omega)$ for $p \in[2,6)$,
every sequence in $H_0^1(\Omega)$ has a subsequence that is Cauchy in the
norm $\lVert \cdot \rVert_{L^p(\Omega)}$, i.e., $\Phi_{k_{mj}}\rightarrow
\Phi_{\infty}$ in $\lVert \cdot \rVert_{L^p(\Omega)}$ norm as
$j \rightarrow \infty$.\\
Then there yields,
\begin{equation}
	\left \{
	\begin{array}{l}
	\displaystyle \lim_{j \rightarrow \infty}\int_{\Omega}V_{ext}(x)\rho_{\Phi_{k_{mj}}}dx=	\int_{\Omega}V_{ext}(x)\rho_{\Phi_{\infty}}dx\\
	\displaystyle \\
	\displaystyle \lim_{j \rightarrow \infty}\int_{\Omega}v_{xc}(\rho_{\Phi_{k_{mj}}})dx=	\int_{\Omega}v_{xc}(\rho_{\Phi_{\infty}})dx\\
	\displaystyle \\
	\displaystyle \lim_{j \rightarrow \infty}\int_{\Omega}V_{Har}(\rho_{\Phi_{k_{mj}}})\rho_{\Phi_{k_{mj}}}dx=\int_{\Omega}V_{Har}(\rho_{\Phi_{\infty}}) \rho_{\Phi_{\infty}}dx.
	\end{array}
	\right.
\end{equation}
Moreover, considering the norm is weakly lower-semicontinuous, there holds
\begin{equation}
	\lim_{j \rightarrow \infty}\inf \lVert \nabla \Phi_{k_{mj}}\rVert_{0,\Omega}
	\geq \lVert \nabla \Phi_{\infty}\rVert_{0,\Omega}.
\end{equation}
Based on the above deduction, we can reach the following result:
\begin{equation}
	\lim_{j \rightarrow \infty}\inf E(\Phi_{k_{mj}})\geq E(\Phi_{\infty}).
\end{equation}
Suppose $U_{\infty}$ is a minimizer of the energy functional in
$V_{\infty}\cap\mathcal{O}^N$. (\ref{appro}) implies that there 
exists a sequence $\{U_j\}_{j \in \mathbb{N}_0}$ such that 
$U_j \in V_{k_{mj}}\cap \mathcal{O}^N$
and $U_j \rightarrow U_{\infty}$ in $\mathcal{H}$. Therefore,
\begin{equation}
	E(U_{\infty})=\lim_{j\rightarrow \infty}E(U_j).
\end{equation}
Since $U_{\infty}$ is a minimizer of the energy functional in
$V_{\infty}\cap \mathcal{O}^N$ and $\Phi_{k_{mj}}$ is also a 
minimizer of the energy functional in $V_{k_{mj}}\cap \mathcal{O}^N$, 
we have
\begin{equation}
	\left \{
	\begin{array}{l}
	\displaystyle E(\Phi_{\infty}) \geq E(U_{\infty})\\
	\displaystyle \\
	\displaystyle E(U_j) \geq E(\Phi_{k_{mj}}),	
	\end{array}
  \right .
\end{equation}
which leads to
\begin{equation}
	\lim_{j \rightarrow \infty} \inf E(\Phi_{k_{mj}}) \geq E(\Phi_{\infty}) \geq
	E(U_{\infty})=\lim_{j \rightarrow \infty}E(U_j)\geq
	\lim_{j \rightarrow \infty} \inf E(\Phi_{k_{mj}}).
\end{equation}
The above relation implies (\ref{firstresult}),
\begin{equation}
	\lim_{j \rightarrow \infty}E(\Phi_{k_{mj}})=E(\Phi_{\infty})=\min_{U \in V_{\infty}\cap \mathcal{O}^N}E(U).
\end{equation}

For the second step, we want to prove all AFE approximations for the
Kohn--Sham gradient flow model converge to Kohn--Sham ground-state 
solutions, i.e.,
\begin{equation}\label{afecon}
	\lim_{k \rightarrow \infty}E_k=\min_{U \in \mathcal{O}^N}E(U).
\end{equation}
Before approaching the above equality,
we first prove the eigenpair $(\Lambda_{\infty},\Phi_{\infty})$ 
solves the weak form of (\ref{kse}).

 $\forall ~\Gamma \in \mathcal{H}$, there holds
 \begin{eqnarray}\label{rela3}
 	\langle H_{\Phi_{\infty}}\Phi_{\infty}-H_{\Phi_{\infty}}\Phi_{\infty},\Gamma \rangle
 	&=& \langle H_{\Phi_{\infty}}\Phi_{\infty}-H_{\Phi_k}\Phi_k,\Gamma \rangle
 	-\langle R_k(\Phi_k), \Gamma\rangle+\langle R_k(\Phi_k), \Gamma\rangle \\
 	&=& \langle H_{\Phi_{\infty}}\Phi_{\infty}- H_{\Phi_k}\Phi_k, \Gamma\rangle
 	-\langle \Lambda_{\infty}\Phi_{\infty}-\Lambda_k \Phi_k, \Gamma \rangle
 	+ \langle R_k(\Phi_k), \Gamma\rangle.   \nonumber
 \end{eqnarray}
For the first term, using the continuity property, we obtain
\begin{eqnarray}
	\langle H_{\Phi_{\infty}}\Phi_{\infty}- H_{\Phi_k}\Phi_k, \Gamma\rangle
	&\lesssim & \lVert H_{\Phi_{\infty}}\Phi_{\infty}- H_{\Phi_k}\Phi_k\rVert_{1,\Omega} \cdot \lVert \Gamma\rVert_{1,\Omega} \nonumber\\
	& \lesssim &\Lambda_k \lVert \Phi_{\infty}-\Phi_k\rVert_{1,\Omega}\cdot \lVert \Gamma\rVert_{1,\Omega} \nonumber \\
	& \lesssim & \lVert \Phi_{\infty}-\Phi_k\rVert_{1,\Omega}\cdot \lVert \Gamma\rVert_{1,\Omega},
\end{eqnarray}
Combine together with (\ref{rela3}), we arrive at
\begin{equation}
	\langle H_{\Phi_{\infty}}\Phi_{\infty}-H_{\Phi_{\infty}}\Phi_{\infty},\Gamma \rangle  \lesssim \lVert \Phi_{\infty}-\Phi_k\rVert_{1,\Omega}\cdot
	\lVert \Gamma\rVert_{1,\Omega}+\langle R_k(\Phi_k).\Gamma \rangle.
\end{equation}
Since $\Phi_k \rightarrow \Phi_{\infty},$ when $k \rightarrow \infty$, 
and $\lim\limits_{k \rightarrow \infty}\langle R_k(\Phi_k),\Gamma \rangle=0$
by employing the Lemma \ref{le:4}. Thus, the above inequality turns to be zero,
i.e.,
\begin{equation}
	\langle H_{\Phi_{\infty}}\Phi_{\infty},\Gamma\rangle=
	\langle \Lambda_{\infty}\Phi_{\infty},\Gamma\rangle,
\end{equation}
which means this limiting eigenpair solves the weak form of (\ref{kse}).
On account of the fact that the gradient flow model is 
derived to obtain the ground state of the Kohn--Sham system, i.e., the 
wavefunction $\Phi_{\infty}$ solves the variational form of the
Kohn--Sham gradient flow model as well. Based on the above 
consideration, we now turn to demonstrate $\Phi_{\infty} \in \Theta$.

From the construction of the finite element subspaces $S_0^{h_k}(\Omega)$,
$k=0,\cdots,\infty$, we can easily get the following result:
$\forall ~\Phi \in \mathcal{H}$, there exists a sequence
$\{U_{h_k}\}_{k=0}^{\infty}$ such that
\begin{equation}
	\lim_{k \rightarrow \infty}\inf \lVert U_{h_k}-\Phi
	\rVert_{1,\Omega}=0,~~\forall~ \Phi \in \mathcal{H},
\end{equation}
which is equivalent to
\begin{equation}
	\lim_{h \rightarrow 0}\inf_{U \in V_h} \lVert U-\Phi
	\rVert_{1,\Omega}=0, ~~\forall~ \Phi \in \mathcal{H}.
\end{equation}
Choosing an initial mesh $\mathcal{T}_0$, such that
\begin{equation}
	E_0\equiv \min_{\Phi_{h_0}\in V_{h_0}\cap \mathcal{O}^N} E(\Phi_{h_0})<
	\min_{(\Lambda,U)\in \mathcal{W}\backslash \Theta}E(U),
\end{equation}
where $\mathcal{W}$ is the set of solutions of the weak form of the 
Kohn--Sham gradient flow system.

Due to $E_k \leq E_0$, it yields
\begin{equation}
	\lim_{k\rightarrow \infty} E(\Phi_{h_k})< \min_{(\Lambda,U)\in \mathcal{W}\backslash \Theta}E(U),
\end{equation}
that is,
\begin{equation}
	E(\Phi_{\infty})= \min_{\Phi\in V_{\infty}\cap \mathcal{O}^N} E(\Phi) < \min_{(\Lambda,U)\in \mathcal{W}\backslash \Theta}E(U).
\end{equation}
Suppose $\tilde{U}$ is a minimizer of the energy functional
in $\mathcal{O}^N$, then there holds
\begin{equation}
	E(\tilde{U})=\min_{U \in \mathcal{O}^N}E(U)\leq E(\Phi_{\infty})<
	\min_{(\Lambda,U)\in \mathcal{W}\backslash \Theta}E(U).
\end{equation}
Furthermore, using the fact that
\begin{equation}
	\lim_{h \rightarrow 0}\inf_{U \in V_h}\lVert U-\bar{U}
	\rVert_{1,\Omega}=0,~~\forall ~\bar{U} \in \mathcal{H}.
\end{equation}
That is to say, for $\tilde{U} \in \mathcal{O}^N$, we can find
$U_{\infty} \in V_{\infty}$ such that
\begin{equation}
	\lVert U_{\infty}-\tilde{U}\rVert_{1,\Omega}=0.
\end{equation}
It can easily be checked that
\begin{equation}
	E(U_{\infty})=E(\tilde{U}).
\end{equation}
As $\Phi_{\infty}$ is the minimizer of the energy functional in
$V_{\infty}\cap \mathcal{O}^N$, it follows that
\begin{equation}
	E(\tilde{U})=\min_{U\in \mathcal{O}^N}E(U)\leq E(\Phi_{\infty})
	\leq E(U_{\infty})=E(\tilde{U})= \min_{U\in \mathcal{O}^N}E(U)
	<\min_{(\Lambda,U)\in \mathcal{W}\backslash \Theta}E(U),
\end{equation}
thus,
\begin{equation}
	E(\Phi_{\infty})=\min_{U\in \mathcal{O}^N}E(U),
\end{equation}
which hints $\Phi_{\infty} \in \Theta$, i.e., 
the equality (\ref{afecon}) holds. This completes the proof.
\end{proof}

So far, we have finished the analysis for spatial and temporal 
convergence, respectively. Now, we have arrived at the main 
conclusion that under some proper conditions,
\begin{equation}
	E(U_k^n) \rightarrow E(U), ~k,n \rightarrow \infty.
\end{equation}
We can guarantee from the initial guess that our algorithm can 
converge to the ground state of a given molecular system from
all-electron calculations.

\section{Numerical experiments}
In this section, focusing on the proposed method for 
the gradient flow model, we verify the stability of the numerical scheme.
The performance of the proposed numerical method will be shown by some 
numerical experiments, which would verify the convergence of our numerical
scheme both spatially and temporally. 

The hardware for the simulations is a Dell 
OptiPlex 7060 workstation with Inter(R) Core(TM) i7--8700 CPU @3.20 GHz 
and 8.00 GB of memory, while the code for the simulations is developed 
based on the iFEM package \cite{chen2009integrated}.

Before introducing the numerical examples, it is noted that 
the ground state for each given atom or molecule is calculated accurately
with our $h$-adaptive finite element method. In our method, the GMRES 
solver is employed for solving the model problem. The time step 
$\Delta t_n$ always be small in the simulation, ranging from 
$1.0 \times 10^{-6}$ to $1.0 \times 10^{-3}$.

It should be mentioned that the adaptive stopping 
criterion strategy is adopted to accelerate the simulations. There is a
phenomenon that the total energy will decrease dramatically only with 
several steps, and then it will take a such long time to correct to the 
lowest energy, which has been verified in \cite{hu2022}. Based on this 
consideration, some standard is needed to judge the decrease rate of 
the total energy, that is,
\begin{equation}\label{stoppcri}
	\frac{E(U^n)-E(U^{n+1})}{\Delta t_n} < \epsilon,
\end{equation}
when satisfies this criterion, we stop the computing 
on the current mesh and call for the $h$-adaptive algorithm to prepare 
the new mesh for the next round computing. Additionally, it should be 
pointed out that the criterion could not be too small.

The following three examples are given to validate our theoretical results.
\begin{example}\label{example1}
  A He atom

  We consider the gradient flow model for a helium atom with orbit $N=1$
  on $\Omega_1 = (-20,20)^3$. We set external potential
  $V_{ext}(x) = -2|x|^{-1}$ and electron density $\rho = 2|u|^2$. The
  position of the nucleus is the origin. The Hartree potential is
  obtained by solving the Poisson equation
  $-\nabla^{2} V_{Har}(x)=4 \pi\rho(x)$ with zero Dirichlet boundary
  conditions. The exchange-correlation energy is obtained by LDA
  \cite{perdew1981self}.

  The energy convergence curve with time and the contour plot of the
  electron density of helium atom are given in Fig.\ref{ex1-norm}.
  Fig.\ref{ex1-norm} (left) shows the energy approximations converge
  monotonically, and the numerical convergence curve is showed to
  agree with the theoretical result, which is very close to the
  referenced value -2.887688 au from \mbox{CCCBDB} \cite{cccbdb}.

  \begin{figure}[hpt]
    \begin{center}
      \includegraphics*[width=0.5\textwidth]{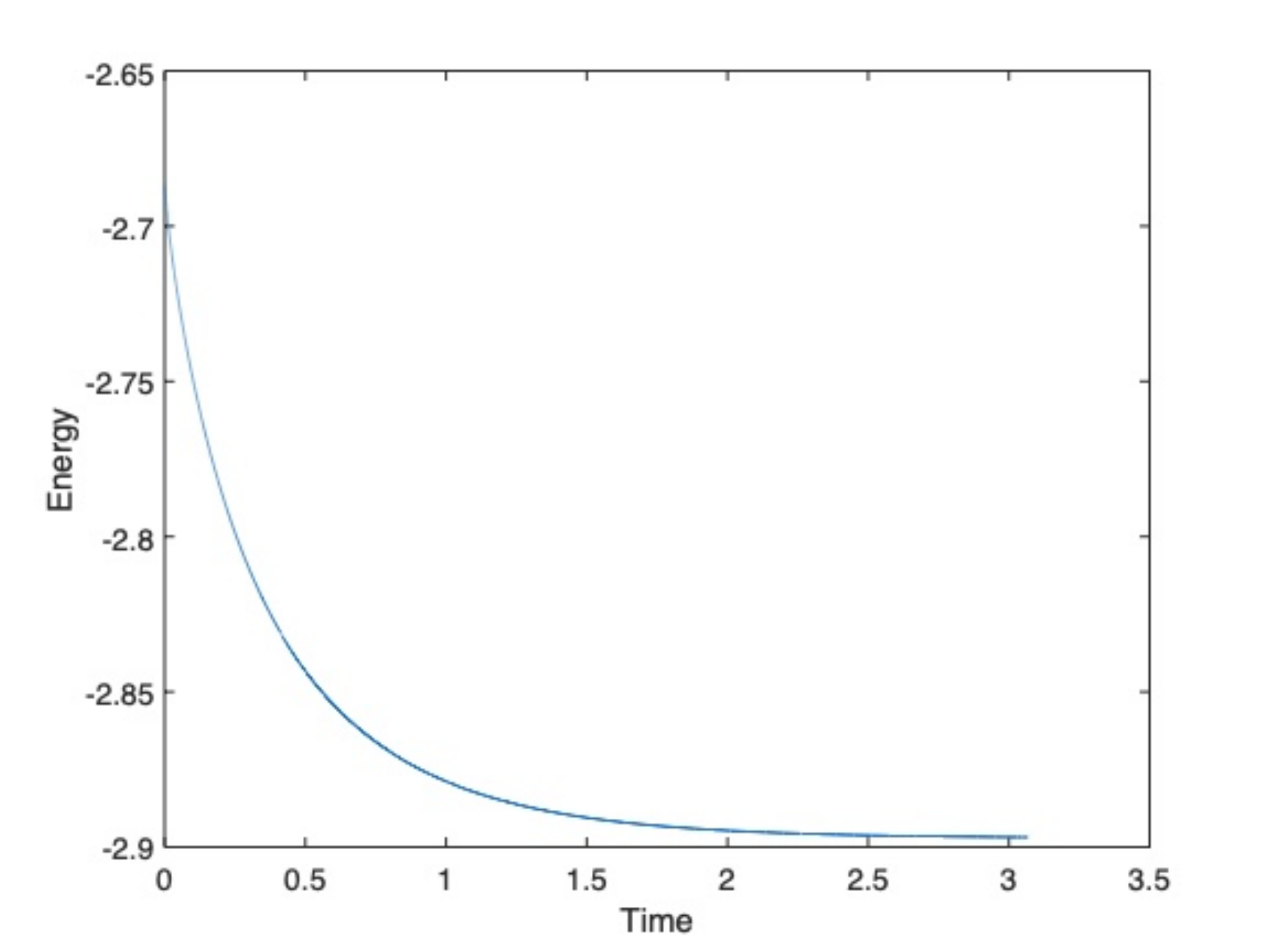}
      \includegraphics*[width=0.35\textwidth]{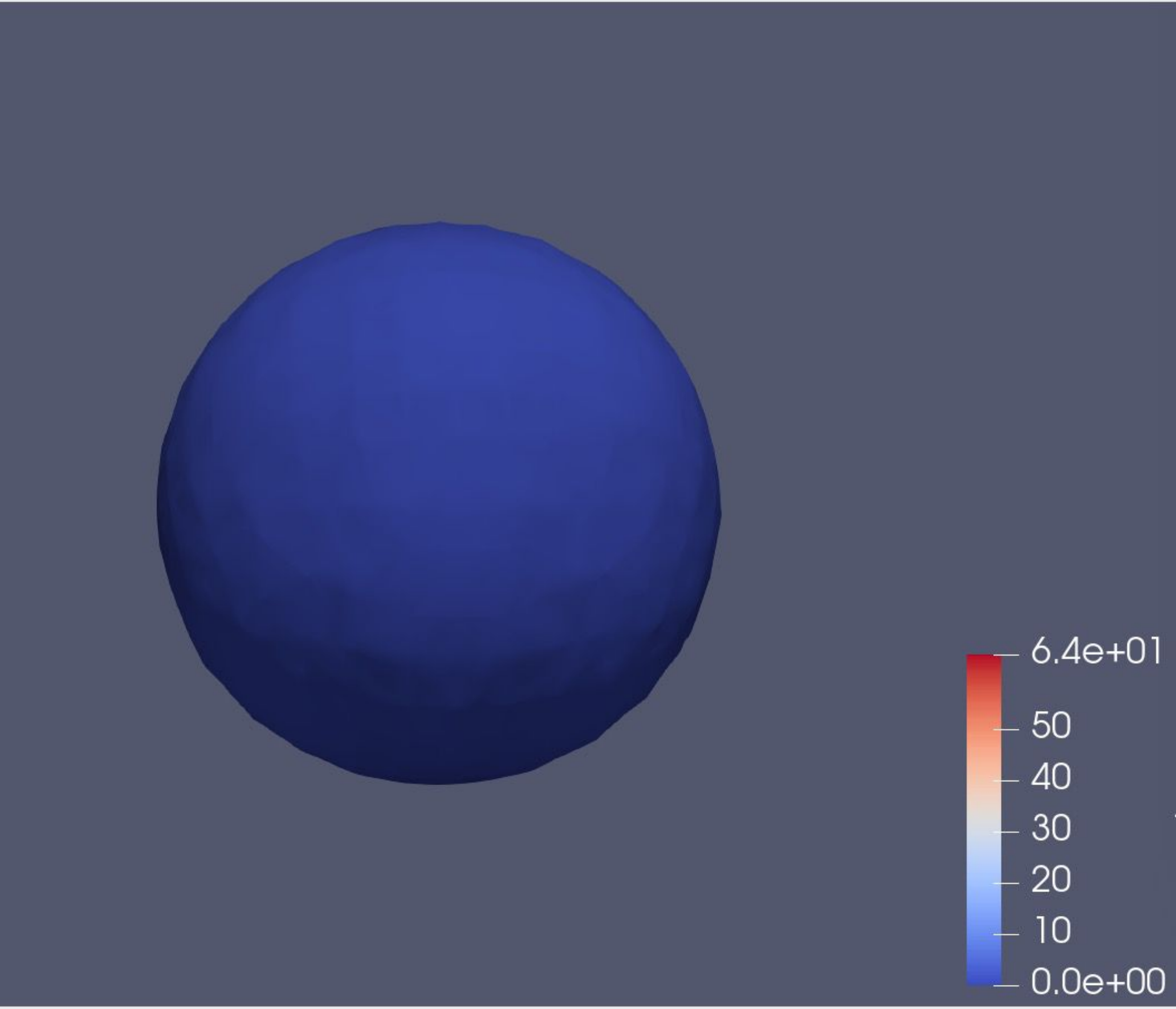}
    \end{center}
    \caption{Left: Convergence of the algorithm with time; Right: The
      contour plot of the electron density of helium
      atom.  \label{ex1-norm}}
  \end{figure}

\end{example}

\begin{example}\label{example2}
  A $H_2$ molecule
  
  We consider the gradient flow model for a $H_2$ molecule with orbits
  number N = 2 on $\Omega_2 = (-20,20)^3 $. In this example, we set
  external potential $V_{ext}(x) = -|x-r_1|^{-1}-|x-r_2|^{-1}$ and
  electron density $\rho = 2(|u_1|^2+|u_2|^2)$.The positions of nuclei
  are as follows: one hydrogen atom $r_1=(-0.7414,0,0)$ and another atom
  $r_2=(0.7414,0,0)$, respectively. The Hartree potential is obtained
  by solving the same Poisson equation in which its boundary
  conditions are given by multipole expansion
  \cite{bao2013numerical}. The evaluation of exchange-correlation
  energy is the same as that in Example \ref{example1}. Similar to the
  helium atom, the energy convergence curve with time and the contour
  plot of the electron density of $H_2$ molecule are given in
  Fig.\ref{ex2}, which converge well and agree with our theoretical
  analysis.
  \begin{figure}[hpt]
    \begin{center}
      \includegraphics[width=0.5\textwidth]{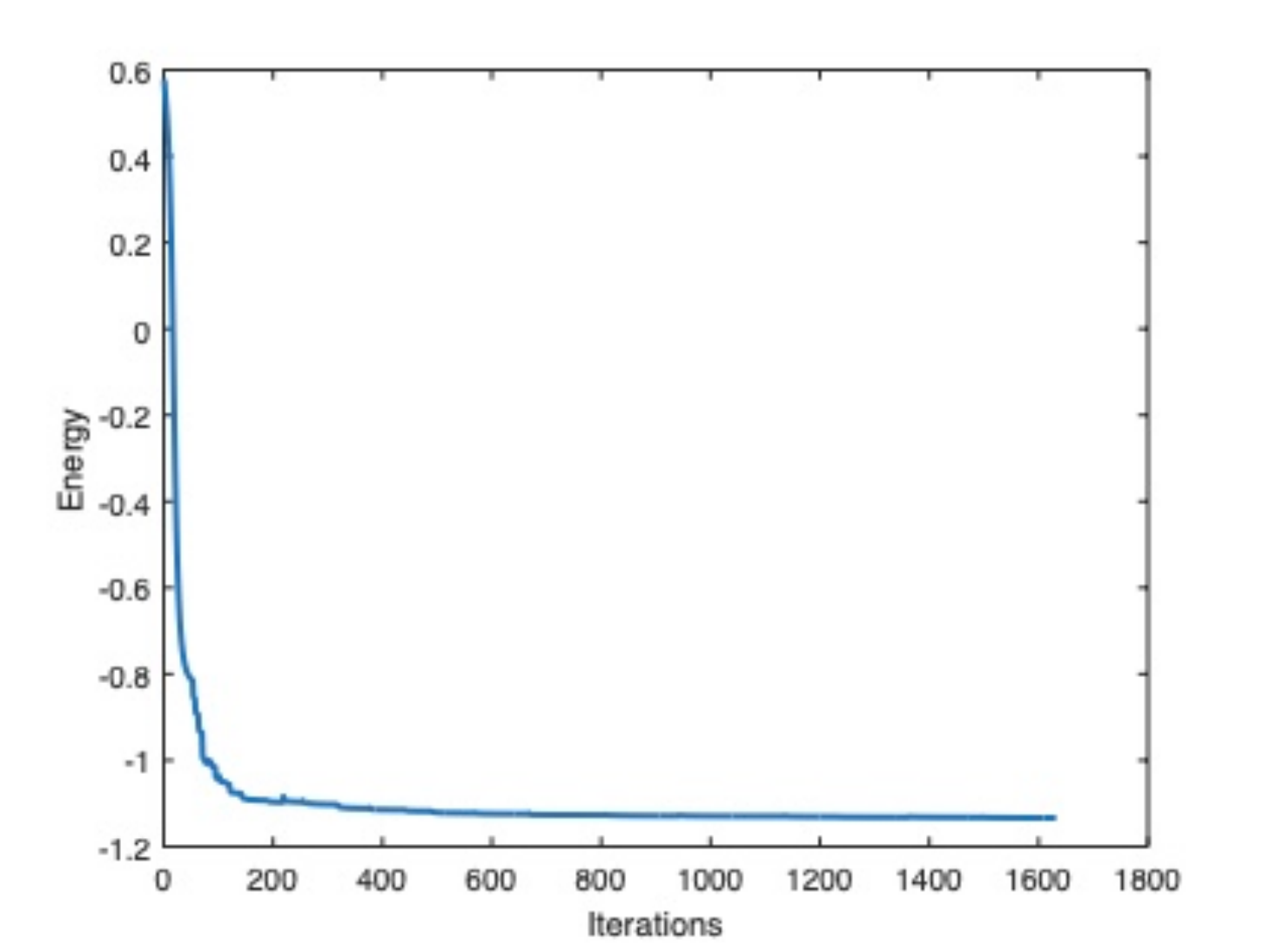}
      \includegraphics[width=0.4\textwidth]{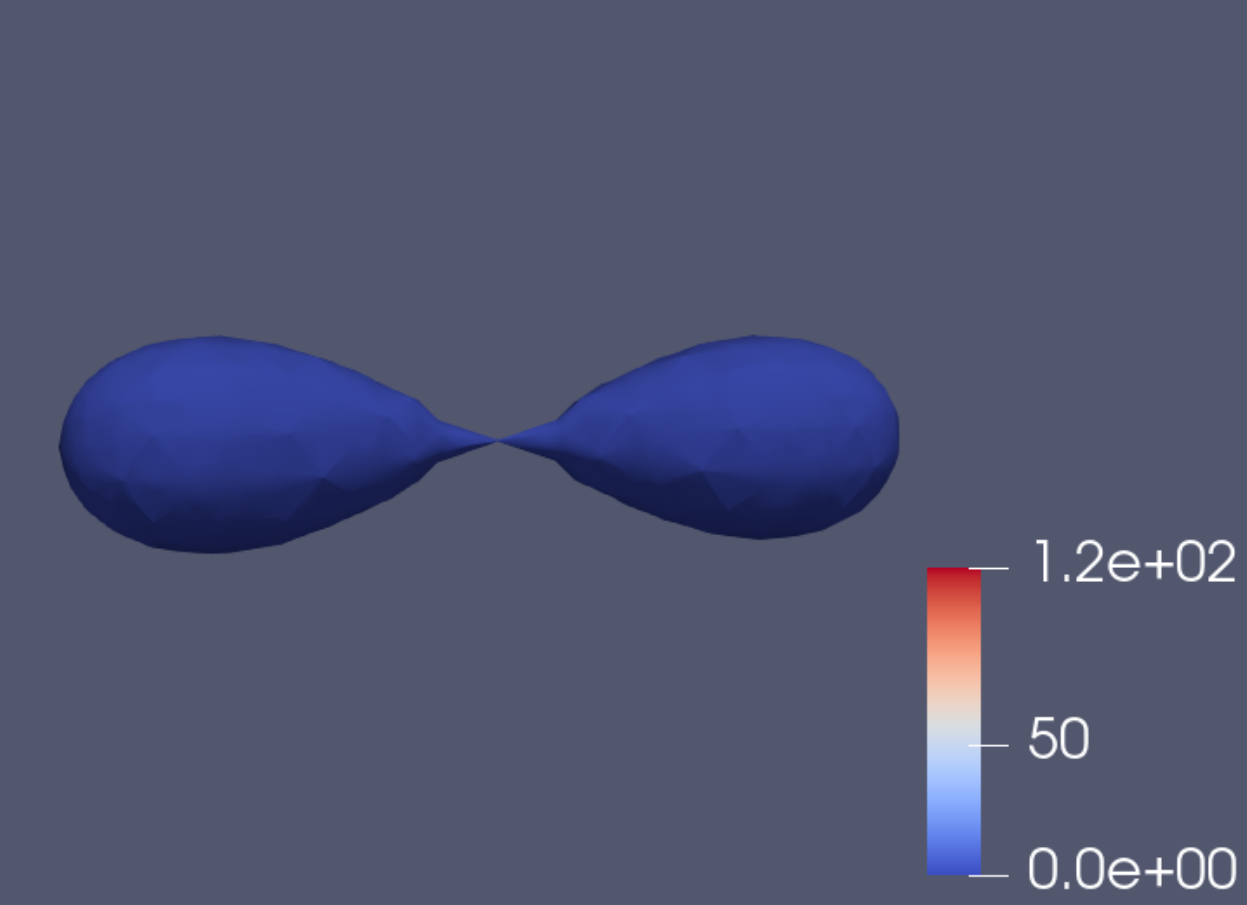}
    \end{center}    
    \caption{Left: The contour plot of the electron density of $H_2$
      molecule; Right: Convergence of the algorithm with mesh
      refining.\label{ex2}}
  \end{figure}
\end{example}

\begin{example}
A LiH molecule

We consider the gradient flow model for a lithium hydride (LiH)
molecule with 2 electron orbitals. The positions of nuclei are as
follows: lithium atom $R_1=(-1.0075,0,0)$ and hydrogen atom
$R_2=(2.0075,0,0)$, respectively. In this example, we set external
potential
$V_{ext}(x)=-3\lvert x-R_1\rvert^{-1}-\lvert x-R_2\rvert^{-1}$ and
electron density
$\rho(x)=2(\lvert u_1(x)\rvert^{2}+\lvert u_2(x)\rvert^{2})$.  The
Hartree potential is obtained by solving the same Poisson equation in
which its boundary conditions are given by multipole expansion
\cite{bao2013numerical}. The evaluation of exchange-correlation
potential is the same as in Example \ref{example1}. We set the
computational domain $\Omega_2=(-10,10)^3$.

Fig.\ref{ex3} shows the results from the ground state simulation for
LiH molecule. The contour of the electron density of $LiH$ molecule is
shown in Fig.\ref{ex3} (top left), and it shows the global
convergence of the total energy of LiH with the adaption refinement of
the mesh as well (top right).  Additionally, the mesh grids around the
molecule with different number of iteration steps increasing from left
to right are also shown in Fig.\ref{ex3} (bottom), from which we can
observe that the regions with large variation of density are resolved
well by mesh grids.

\begin{figure}[hpt]
  \begin{center}
    \includegraphics[width=0.4\textwidth]{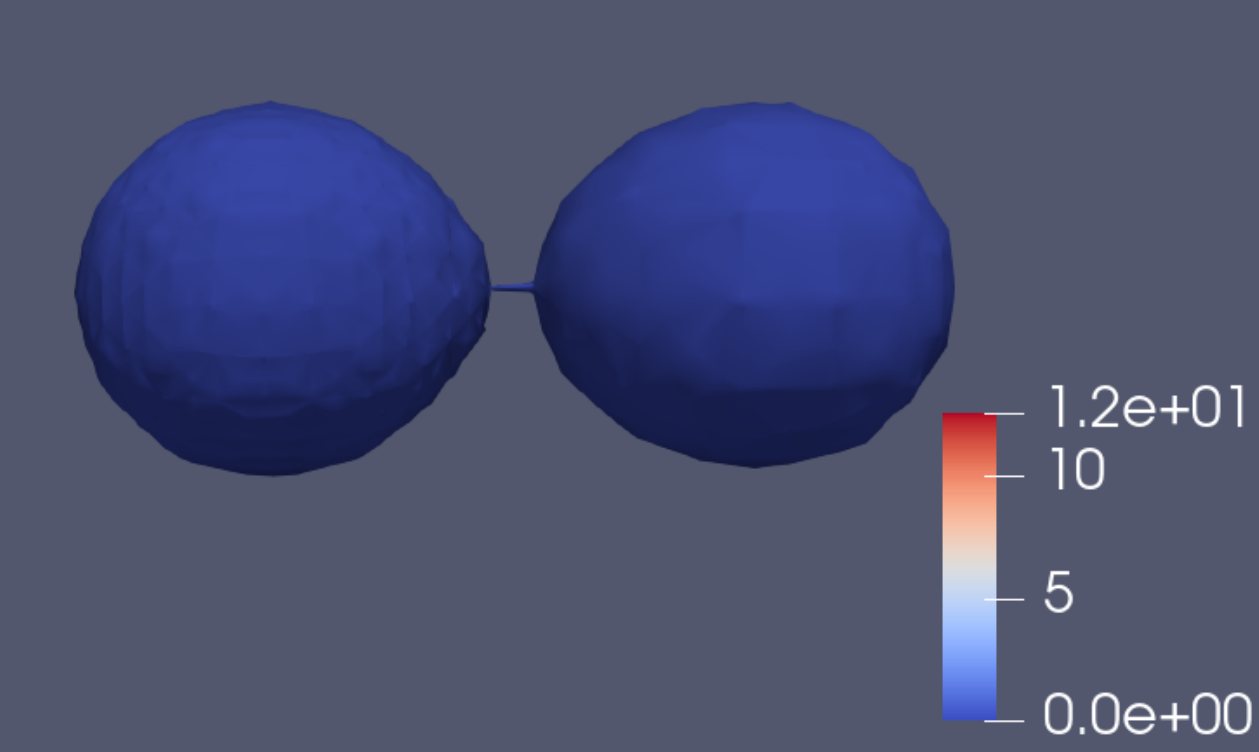}
    \includegraphics[width=0.4\textwidth]{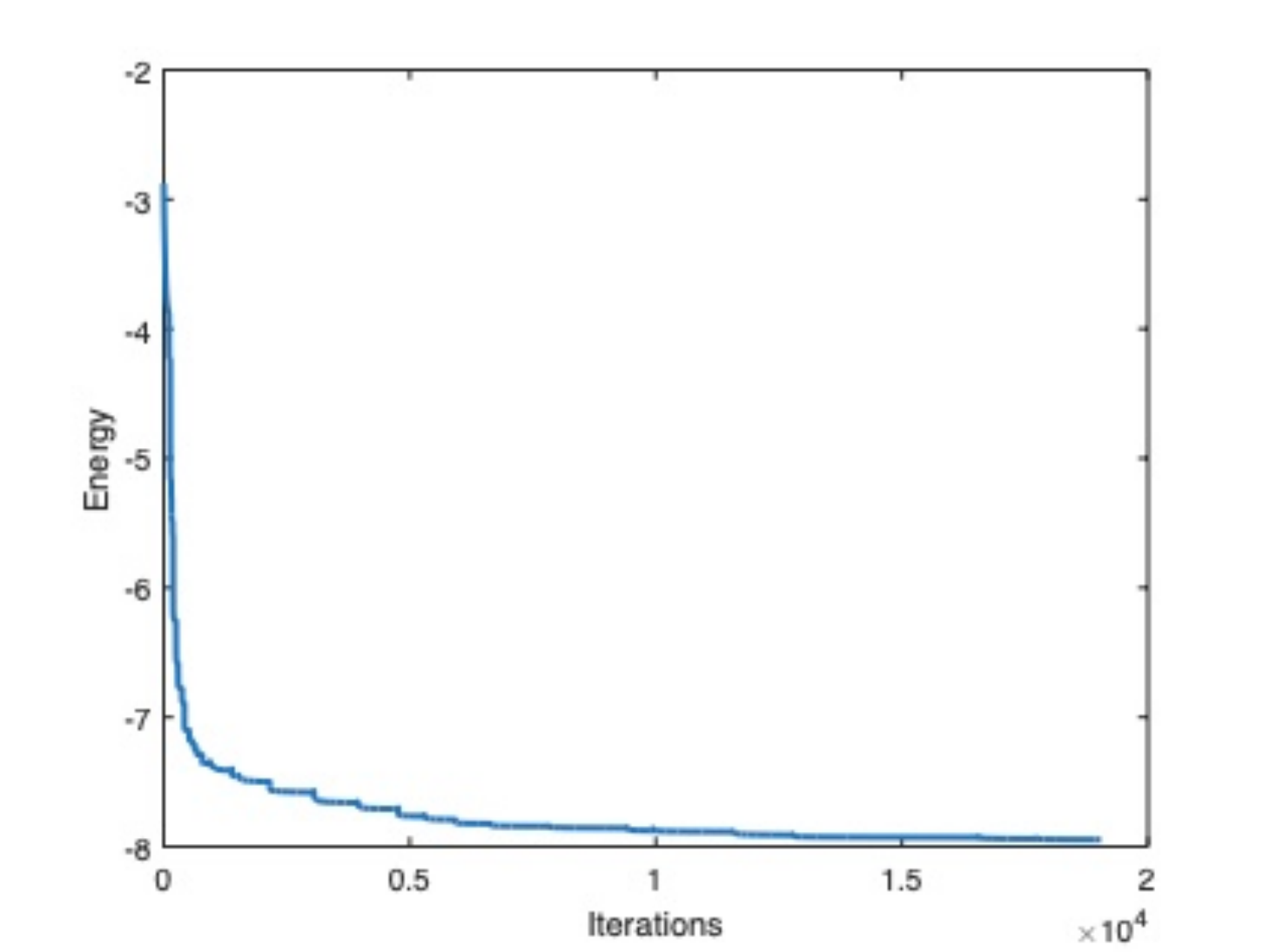}\\
    \includegraphics[width=0.32\textwidth]{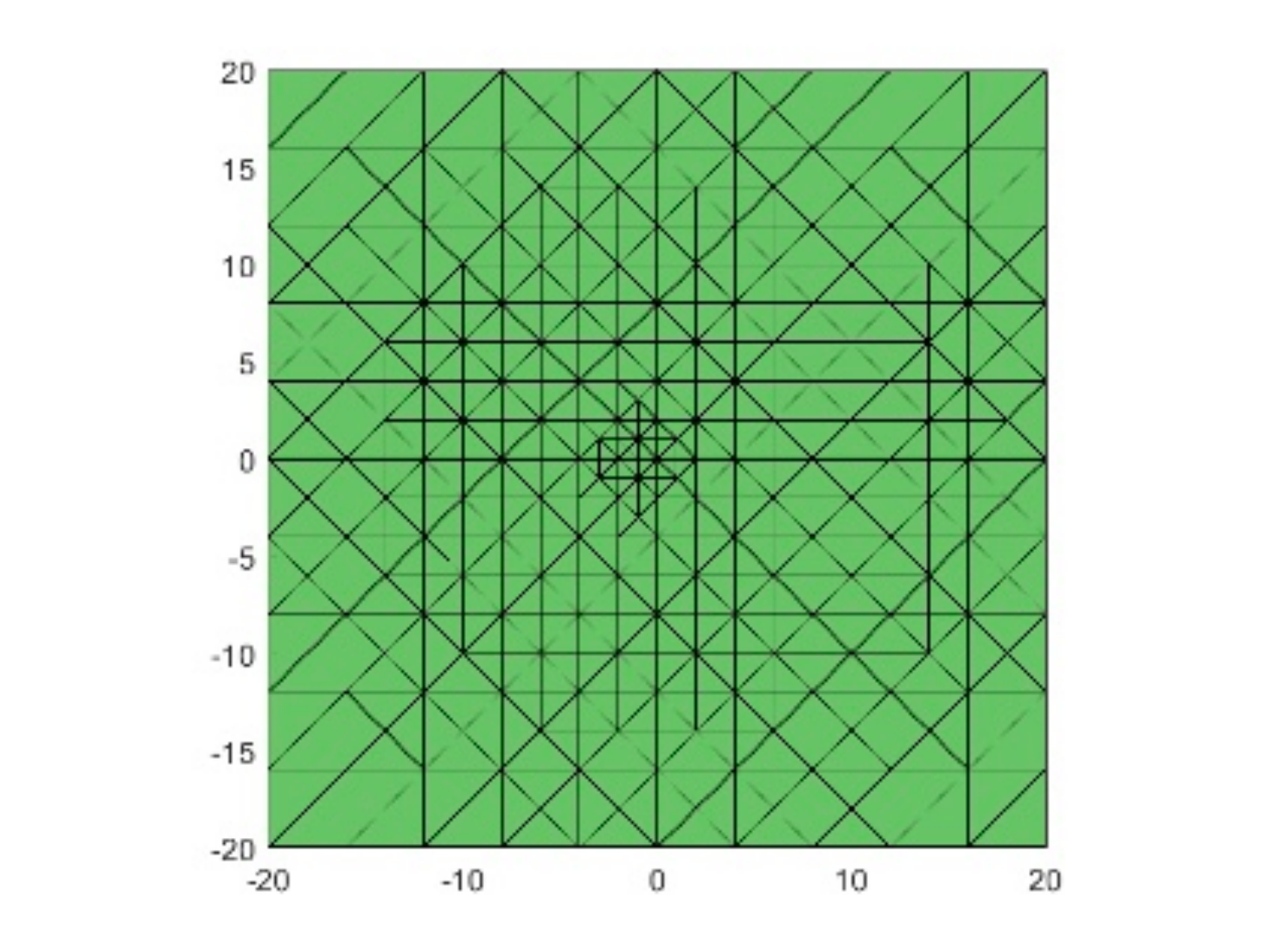}
    \includegraphics[width=0.32\textwidth]{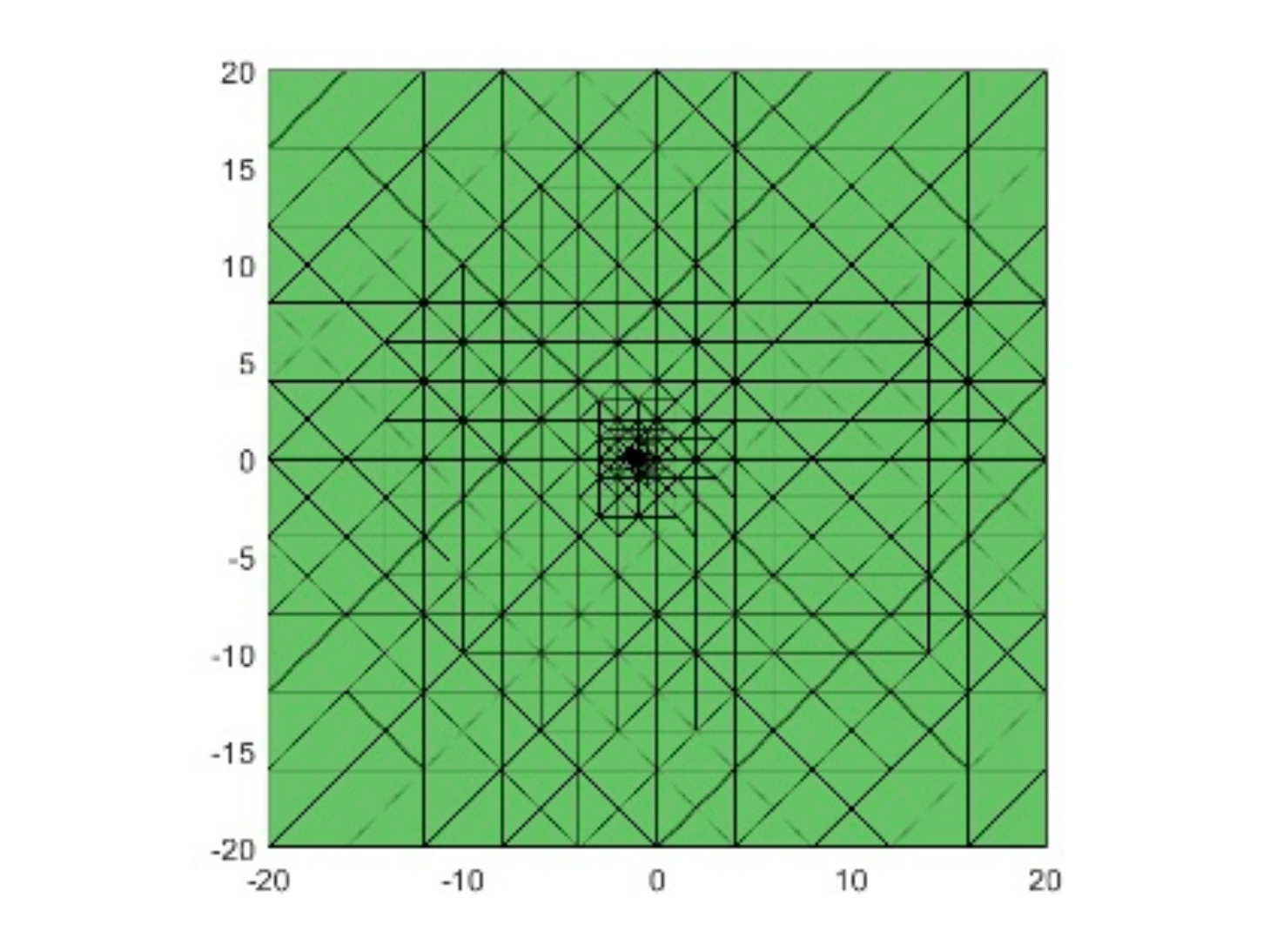}
    \includegraphics[width=0.32\textwidth]{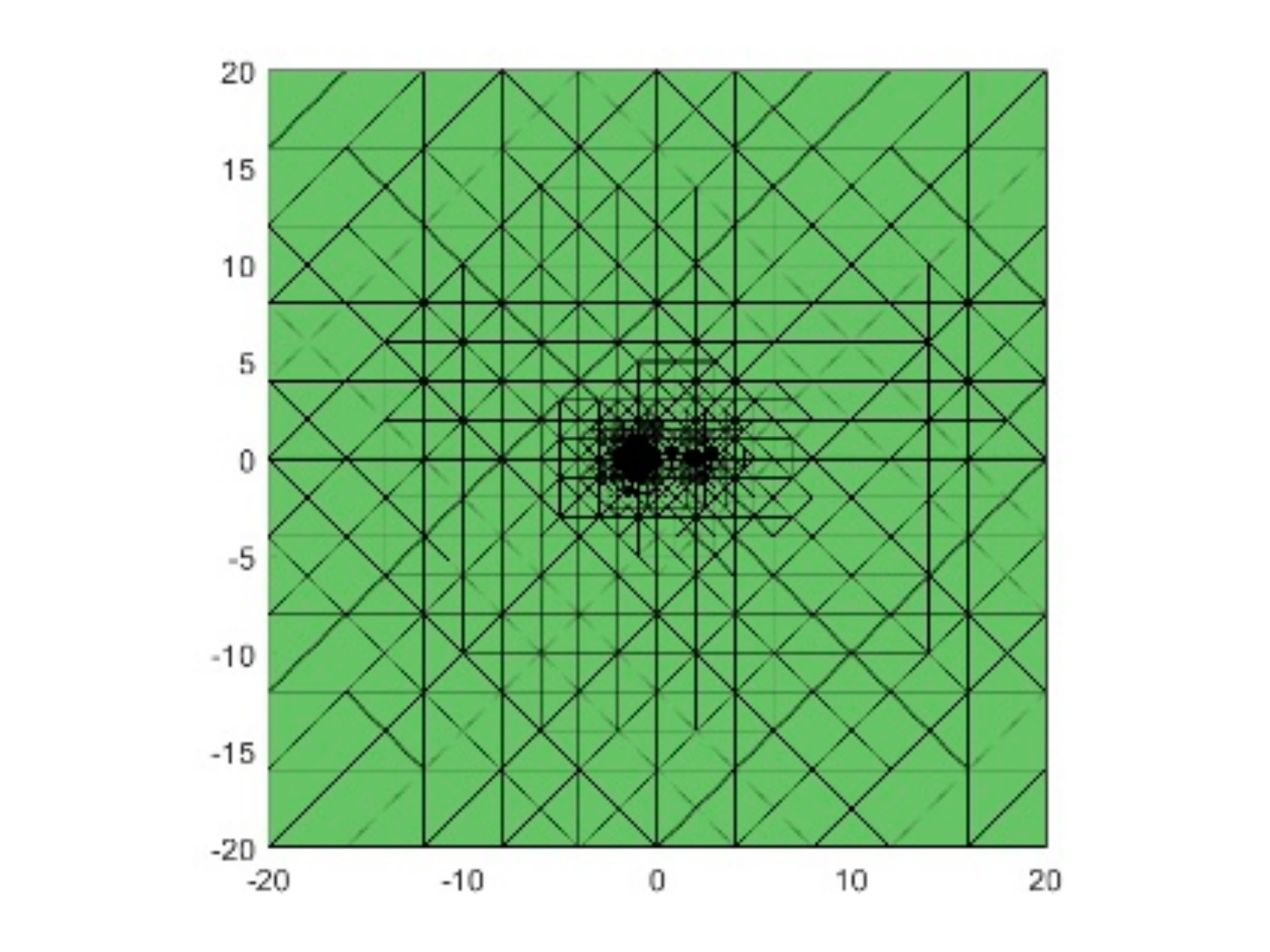}
  \end{center}
  \caption{The results for LiH. Top left: the contour of the electron
    density. Top right: the global convergence history of LiH. Bottom:
    the mesh grids around the LiH molecule, of which the number of
    iteration steps increase from left to right.\label{ex3}}
\end{figure}

\end{example}

\section{Conclusion}
This paper is devoted to show the convergence of 
the proposed numerical scheme for the Kohn--Sham gradient flow problem 
both spatially and temporally in a unified framework. Spatially, the
convergence of the AFE approximation of the gradient flow 
model has been demonstrated by means of the boundedness of the total 
energy in the all-electron calculation.
Temporally, motivated by \cite{dai2020gradient}, a linearized numerical
scheme for the gradient flow model is employed in this paper, which is
orthonormality preserving and has been proved to be convergent as well.
Concerned with the efficiency issue, an $h$-adaptive mesh method is
developed and a recovery type a \emph{posteriori} error estimation
technique is employed in the $h$-adaption scheme. Numerical examples 
successfully show the convergence of the proposed numerical scheme and 
confirm our theoretical results very well.


\section*{Acknowledgments}
We would like to thank the anonymous reviewers for their helpful
comments. The third author was partially funded by Hunan National
Applied Mathematics Center of Hunan Provincial Science and Technology
Department (2020ZYT003), and Excellent youth project of Hunan
Education Department (19B543).  The fourth author was partially
supported by National Natural Science Foundation of China (Grant
Nos. 11922120 and 11871489), FDCT of Macao S.A.R. (Grant
No. 0082/2020/A2), MYRG of University of Macau (MYRG2020-00265-FST)
and Guangdong-Hong Kong-Macao Joint Laboratory for Data-Driven Fluid
Mechanics and Engineering Applications (2020B1212030001).

\bibliographystyle{IEEEtran}

\end{document}